\documentclass[12pt]{amsart}

\usepackage{graphicx} 
\usepackage{scrextend}
\usepackage{amsfonts, amsmath, amssymb,amsthm}  
\usepackage{times,enumerate}
\usepackage{mathdots}%
\usepackage{nccmath}
\usepackage{mathtools}
\usepackage{ulem}
    {\end{pmatrix}\end{medsize}}%
\usepackage{dsfont,bbm}
\usepackage{rotating}
\usepackage{lscape}
\usepackage{url}
\usepackage{multicol}
\usepackage{thmtools, thm-restate}
\usepackage{blkarray}
\usepackage[margin=2.5cm]{geometry}
\usepackage{enumitem}
\usepackage[pagebackref]{hyperref}
\usepackage[nameinlink,capitalise]{cleveref}
\usepackage{cancel}
\usepackage{multirow}
\usepackage{pifont}
\usepackage{tikz}
\usepackage{array}
\usepackage{booktabs}
\usepackage{tabularx}
\usetikzlibrary{topaths}
\tikzstyle{every picture}+=[remember picture,inner xsep=0,inner ysep=0.25ex]
\usetikzlibrary{calc}
\usepackage{kbordermatrix}
\renewcommand{\kbldelim}{(}
\renewcommand{\kbrdelim}{)}
\renewcommand{\kbrowstyle}{\displaystyle}
\renewcommand{\kbcolstyle}{\displaystyle}
\def\VR{\kern-\arraycolsep\strut\vrule &\kern-\arraycolsep}
\def\vr{\kern-\arraycolsep & \kern-\arraycolsep}
\makeatletter
\newcommand*{\sublabel}[1]{%
    \let\old@currentlabel\@currentlabel%
    \renewcommand{\@currentlabel}{\theenumii}%
    \label{#1}%
    \let\@currentlabel\old@currentlabel%
}
\makeatother

\graphicspath{ {./figs/} }

\DeclareFontFamily{U}{mathx}{\hyphenchar\font45 }
\DeclareFontShape{U}{mathx}{m}{n}{
	<5><6><7><8><9><10><10.95><12><14.4><17.28><20.74><24.88> mathx10
}{}
\DeclareSymbolFont{mathx}{U}{mathx}{m}{n}
\DeclareMathAccent{\widecheck}{0}{mathx}{"71}

\makeatletter
\def\widebreve{\mathpalette\wide@breve}
\def\wide@breve#1#2{\sbox\z@{$#1#2$}%
     \mathop{\vbox{\m@th\ialign{##\crcr
\kern0.08em\brevefill#1{0.8\wd\z@}\crcr\noalign{\nointerlineskip}%
                    $\hss#1#2\hss$\crcr}}}\limits}
\def\brevefill#1#2{$\m@th\sbox\tw@{$#1($}%
  \hss\resizebox{#2}{\wd\tw@}{\rotatebox[origin=c]{90}{\upshape(}}\hss$}
\makeatletter

\newcommand{\RR}{\mathbb R}

\newcommand{\NN}{\mathbb N}

\newcommand{\cZ}{\mathcal Z}

\newcommand{\benu}{\begin{enumerate}}
\newcommand{\eenu}{\end{enumerate}}
\newcommand{\bop}{\begin{opomba}}
\newcommand{\eop}{\end{opomba}}

\allowdisplaybreaks

\newtheorem{theorem}{Theorem}[section]
\newtheorem{corollary}[theorem]{Corollary}
\newtheorem{lemma}[theorem]{Lemma}
\newtheorem{proposition}[theorem]{Proposition}

\theoremstyle{definition}

\newtheorem{remark}[theorem]{Remark}

\definecolor{green-new}{rgb}{0.0, 0.5, 0.0}

\definecolor{cyan}{rgb}{0.0, 0.8, 1.0}

\numberwithin{equation}{section}

\begin{document}

\title{Truncated Moment Problems and the Extension Property on Monomial Curves}

\author[R. Nailwal]{Rajkamal Nailwal${}^{1}$}
\address{Rajkamal Nailwal, Indian Institute of Technology Kanpur, Kanpur, India}
\email{raj1994nailwal@gmail.com}
\thanks{${}^1$ Supported by ANRF-ARG Project:  ANRF /ARG/2025/001228/MS} 

\author[A. Zalar]{Alja\v z Zalar${}^{2}$}
	\address{Alja\v z Zalar, 
        University of Ljubljana, 
		Faculty of Computer and Information Science \& 
		Faculty of Mathematics and Physics, \&
		Institute of Mathematics, Physics and Mechanics, Ljubljana, Slovenia.}
\email{aljaz.zalar@fri.uni-lj.si}
\thanks{${}^2$ Supported by the ARIS
research core funding No.\ P1-0288 and grants No.\ J1-50002, J1-60011, J1-70017.} 

\author[I. Zobovi\v c]{Igor Zobovi\v c${}^{3}$}
\address{I.\ Zobovi\v c, 
Institute of Mathematics, Physics and Mechanics, Ljubljana, Slovenia.}
\email{igor.zobovic@imfm.si}
\thanks{${}^3$ Supported by the ARIS
research core funding No.\ P1-0288.}

\subjclass[2020]{Primary 44A60, 47A57; Secondary 14P05, 13J30, 11E25.}
\keywords{Truncated moment problem; representing measure; positive semidefinite extension; moment matrix; monomial curve; sums of squares.}



\begin{abstract}
In several papers, Stochel and Szafraniec studied moment problems on algebraic sets from an operator-theoretic perspective, investigating when positive definite sequences satisfying polynomial relations admit representing measures. Within this framework, Stochel introduced type A sets, and Bisgaard classified the plane curves defined by relations between two monomials that have this property.
Curto and Fialkow introduced a stronger, truncated version of the type A property, requiring that the existence of a positive semidefinite extension of prescribed degree guarantees the existence of a representing measure. Motivated by Bisgaard’s classification, we determine which plane curves defined by relations between two monomials satisfy this extension property. In the affirmative cases, we obtain explicit bounds on the required extension degree. In the negative cases, we construct truncated sequences that admit positive semidefinite extensions of arbitrarily high order but have no representing measure supported on the curve. These constructions yield explicit polynomials that are nonnegative on the corresponding curves but are not sums of squares in their coordinate rings. In the affirmative cases, we also derive explicit degree bounds for sums-of-squares certificates of strictly positive polynomials.
\end{abstract}

\maketitle

\section{Introduction}
In \cite{StochelSzafraniec94}, Stochel and Szafraniec studied moment problems on algebraic sets from an operator-theoretic perspective, relating the existence of representing measures for positive definite sequences satisfying polynomial relations to normal extensions of algebraic operators. The operator-theoretic approach to moment problems was
further developed in \cite{JS05}, where vector and operator
multidimensional moment problems with jointly subnormal solutions
were characterized in terms of their initial data. Related connections between polynomial relations modulo an
ideal, multivariate three-term recurrences, and orthogonality with
respect to measures supported on algebraic sets were developed in
\cite{css05}. Within this framework, Stochel introduced in \cite{Sto92} the notions of algebraic sets of type A and type B. More precisely, let
$
\cZ_p:=\{x\in\mathbb{R}^d:p(x)=0\}
$
be the real zero set of a polynomial \(p\). The set \(\cZ_p\) is said to be of \textit{type A} if every linear functional on \(\mathbb{R}[x_1,\ldots,x_d]\) that is nonnegative on squares and vanishes on the principal ideal \((p)\) admits a positive representing measure supported on \(\cZ_p\). It is said to be of \textit{type B} if the same conclusion holds for every such functional vanishing on the vanishing ideal
\[
I(\cZ_p):=
\{P\in\mathbb{R}[x_1,\ldots,x_d]:
P(x)=0\ \text{for every }x\in\cZ_p\}.
\]
Using the Real Nullstellensatz, Bisgaard and Stochel proved in \cite[Section~3]{bs16} that the type A and type B properties coincide. Consequently, the type A property depends only on the zero set \(\cZ_p\), not on the choice of its defining polynomial \(p\). 

Stochel also proved in \cite{Sto92} that several classes of plane algebraic curves are of type A, including curves of the forms \(y=q(x)\) and \(yq(x)=1\), and every plane algebraic curve of degree at most \(2\). Using semigroup theory, Bisgaard characterized in \cite{Bis98} which plane curves defined by relations between two monomials,
$
x^{i_1}y^{j_1}=x^{i_2}y^{j_2},
$
are of type A. Moreover, Schm\"udgen's result \cite{s42} implies that every bounded algebraic set \(\cZ_p\) is of type A.

More generally, let
$
A_p:=\mathbb{R}[x_1,\ldots,x_d]/(p).
$
By \cite[Theorem~4.3]{Sto92}, the set \(\cZ_p\) is of type A if and only if the image in \(A_p\) of the cone of sums of squares in \(\mathbb{R}[x_1,\ldots,x_d]\) is dense in the cone of positive elements of \(A_p\), with respect to any locally convex topology on \(A_p\) for which every linear functional is continuous.

Equivalently, $\cZ_p$ is of type~A if and only if the quadratic module generated by
$
Q:=\{1,p,-p\}
$
(see \eqref{def:quadratic-module}) has the \textit{strong moment property} in the sense of \cite[Definition~13.6]{sch17}. The classification of real algebraic curves with the strong moment property follows from results of Scheiderer and Plaumann \cite{s57,p51}; see, for example, \cite[Theorem~3.13 and Proposition~5.2]{p51}.

The present paper is motivated by the study of plane curves defined by relations between two monomials from the perspective of truncated moment theory. A theorem of Stochel \cite{Sto01} shows that solvability of all truncations of a full moment sequence implies solvability of the corresponding full moment problem. Thus, if every truncation of a given sequence has a representing measure supported on a fixed algebraic set, then the full sequence also has a representing measure supported on that set. More general extension criteria for positive definite mappings
on symmetric subsets of $*$-semigroups, including applications to
truncated moment problems, were established in \cite{css11}. For
another application of moment methods in operator theory, namely a
two-moment characterization of spectral measures on the real line,
see \cite{PS23}.

In \cite{CF08}, Curto and Fialkow introduced a strengthening of the type~A property for truncated moment sequences, namely the extension property $(S_{n,k})$; see \eqref{def-S-n-k-property} below. In this framework, the generating families associated with a line, a circle, or an ellipse satisfy $(S_{n,1})$, whereas those associated with a parabola or a hyperbola satisfy $(S_{n,2})$. Moreover, Fialkow proved in \cite{f1} that suitable generating families for curves of the forms $y=q(x)$ and $yq(x)=1$ satisfy $(S_{n,k})$ for sufficiently large $k$. Improved bounds on $k$ were subsequently obtained in \cite{z4}. Property $(S_{n,k})$ also has important consequences in real algebraic geometry: it yields degree-bounded weighted sums-of-squares representations for polynomials that are strictly positive on the underlying semialgebraic set.

In this paper, we investigate property \((S_{n,k})\) for the same
families of monomial plane curves that appear in Bisgaard's
classification of type~A curves. More precisely, we consider
polynomials \(p\) that are monomial multiples of binomials and
determine when the associated generator family \(\{1,p,-p\}\)
satisfies property \((S_{n,k})\) for some finite \(k\). Our results,
together with the previously known cases, yield the classification
summarized in Table~\ref{tab:classification}. The table records
whether a finite extension parameter exists, gives an admissible value
of \(k\) in each affirmative case, and indicates the result in which
the corresponding assertion is established.  

In the negative cases, our constructions also produce explicit
polynomials that are nonnegative on the underlying curves but are not
sums of squares in the corresponding coordinate rings. In the
affirmative cases, the extension bounds yield degree-bounded
positivity certificates. More precisely, if \(\{1,p,-p\}\) satisfies
property \((S_{n,k})\), then every polynomial
\(f\in\mathcal P_{2n}\) that is strictly positive on \(\cZ_p\)
admits a representation
$
        f=\sum_{j=1}^r g_j^2+ph,
$
where \(g_1,\ldots,g_r,h\in\mathbb R[x,y]\) and
$\deg g_j^2\le2(n+k)$,
$\deg(ph)\le2(n+k)$.
This follows from \cite[Theorem~1.5(i)]{CF08}; explicit versions are
given in Corollaries~\ref{cor-after-4-1},
\ref{cor-after-4-4}, and~\ref{cor-after-4-6}. Thus, our results both
extend the theory of truncated moment problems on monomial curves and
provide explicit information about sums-of-squares representations,
including obstructions in the negative cases and degree bounds in the
affirmative cases.

\begin{table}[ht] \centering \caption{Classification of the generator families $\{1,p,-p\}$ with respect to property $(S_{n,k})$, where $n\geq\deg p$. Unless stated otherwise, $\ell\in \NN$. In the rows involving $\ell_1,\ell_2$, we assume $\ell_1,\ell_2\in\mathbb{N}\setminus\{1\}$ and $\gcd(\ell_1,\ell_2)=1$. }\label{tab:classification} \small \renewcommand{\arraystretch}{1.4} \setlength{\tabcolsep}{2pt} \begin{tabularx}{\textwidth}{ @{} >{\raggedright\arraybackslash}p{6.0cm} @{\hspace{5pt}} >{\centering\arraybackslash}p{1.15cm} @{\hspace{5pt}} >{\centering\arraybackslash}X @{\hspace{5pt}} >{\raggedright\arraybackslash}p{3.1cm} @{} } \toprule \textbf{Defining polynomial \(p(x,y)\)} & \textbf{Finite \(k\)} & \textbf{Admissible values of \(k\)} & \textbf{Result} \\ \midrule

\addlinespace[0.25em]

\(x-y^\ell\) or \(y-x^\ell\)
&
Yes
&
\(\displaystyle k\ge\max\{\ell-1,1\}\)
&
\cite[Th.~3.1]{z4}
\\
\addlinespace[0.35em]

\(y(x-y^\ell)\) or \(x(y-x^\ell)\)
&
Yes
&
\(\displaystyle k\ge\ell+\max\{\ell-1,2\}\)
&
Th.~\ref{theo:case-5}
\\
\addlinespace[0.35em]

\(y(y-x^\ell)\) or \(x(x-y^\ell)\),\; $\ell\geq 2$
&
No
&
\textemdash
&
Th.~\ref{theo:case-6}
\\
\addlinespace[0.35em]

\(xy(y-x^\ell)\) or \(xy(x-y^\ell)\),\; $\ell\geq 2$
&
No
&
\textemdash
&
Th.~\ref{theo:case-7}
\\

\midrule

\addlinespace[0.25em]

\(\begin{gathered}
q(x,y)(x^{\ell_1}-y^{\ell_2}),\\[-0.1em]
q\in\{1,x,y,xy\}
\end{gathered}\)
&
No
&
\textemdash
&
Ths.~\ref{theo:case-1}, \ref{theo:case-1.2}, and~\ref{theo:case-1.3}\\

\midrule

\addlinespace[0.25em]

\(xy^\ell-1\) or \(x^\ell y-1\)
&
Yes
&
\(\displaystyle k\ge\ell+1\)
&
\cite[Th.~4.1]{z4}
\\
\addlinespace[0.35em]

\(x^{\ell_1}y^{\ell_2}-1\)
&
Yes
&
\(\displaystyle
k\ge
\ell_1+\ell_2-1\)
&
Th.~\ref{theo:case-2}
\\
\addlinespace[0.35em]

\(\begin{gathered}
q(x,y)(x^{\ell_1}y^{\ell_2}-1),\\[-0.1em]
q\in\{x,y,xy\}
\end{gathered}\)
&
Yes
&
\(\displaystyle
k\ge
2(\ell_1+\ell_2)\)
&
Ths.~\ref{theo:case-3}, \ref{theo:case-4}
\\

\bottomrule
\end{tabularx}
\end{table}

\begin{remark}
\label{rem:classification-observations}
We record several observations concerning the classification in
Table~\ref{tab:classification}.
\begin{enumerate}[leftmargin=*]
\item
For a fixed \(n\in\NN\), the existence of a finite
\(k\in\NN_0\) such that \(\{1,p,-p\}\) satisfies property
\((S_{n,k})\) depends only on the zero set \(\cZ_p\), rather than on
the particular polynomial used to define it; see
Remark~\ref{rem:powers-irrelevant}. This explains why we restrict our
attention to defining polynomials in which every irreducible factor
occurs with multiplicity one.

\item
The assumption \(n\ge\deg p\) ensures that the column \(p(X,Y)\)
occurs in \(M(n)\). If a positive extension satisfies the localizing
conditions associated with \(p\) and \(-p\), then
$
        p(X,Y)=0,
$
together with all column relations obtained from this identity by
recursive generation whenever the relevant products are defined. If
\(n<\deg p\), then the column \(p(X,Y)\) does not occur in \(M(n)\),
and the corresponding cases require a separate analysis. We therefore
restrict our attention to the principal setting \(n\ge\deg p\).
\end{enumerate}
\end{remark}

\subsection{Reader's guide}

The paper is organized as follows. In
Section~\ref{sec:preliminaries}, we introduce the necessary notation
and recall the basic facts about truncated moment problems, moment and
localizing matrices, and property \((S_{n,k})\).

In Section~\ref{sec:nonSnk}, we treat the negative cases summarized
in Table~\ref{tab:classification}. For each of these
families, we construct truncated sequences having extension
properties of every finite order but admitting no representing
measure on the corresponding algebraic set. These constructions also
yield explicit polynomials that are nonnegative on the curves but are
not sums of squares in their coordinate rings.

In Section~\ref{sec:Snk}, we treat the affirmative cases summarized
in Table~\ref{tab:classification}. We establish property
\((S_{n,k})\) for the displayed admissible extension parameters by
reducing the corresponding truncated moment problems to
one-dimensional moment problems. As a consequence, we obtain explicit
degree bounds for sums-of-squares representations of polynomials that
are strictly positive on the curves.

\section{Preliminaries}
\label{sec:preliminaries}

In this section, we recall the basic concepts and results from the theory of truncated moment problems that will be used throughout the paper, including moment and localizing matrices, recursively generated moment matrices, and property \((S_{n,k})\).\\

Let $\NN := \{1, 2, 3, \ldots\}$ and $\NN_0 := \NN \cup \{0\}$. For $\alpha:=(\alpha_1,\ldots,\alpha_d)\in \NN_0^d$, we write $|\alpha|:=\alpha_1+\cdots+\alpha_d$. 
Let 
\[
\beta \equiv \beta^{(2n)} = \{\beta_\alpha\}_{\alpha\in \NN_0^d,\; |\alpha|\le 2n}
\]
be a truncated multivariate sequence of degree $2n$ in $d$ real variables, and let
\[
\mathcal P_n 
:=\{p \in \mathbb{R}[x_1,\dots,x_d] : \deg p \le n\}.
\]
The associated Riesz functional $L_\beta : \mathcal P_{2n} \to \mathbb{R}$ is defined by
\[
L_\beta\Big(\sum_{|\alpha|\le 2n} a_\alpha x^\alpha\Big)
:= \sum_{|\alpha|\le 2n} a_\alpha \beta_\alpha.
\]
Let \[ \mathcal B_n := \{x^\alpha:\alpha\in\mathbb N_0^d,\ |\alpha|\le n\} \] denote the monomial basis of \(\mathcal P_n\), ordered first by total degree and then lexicographically within each degree.

The moment matrix $M(n) \equiv M(n;\beta)$ is the symmetric matrix indexed by $\mathcal B_n \times \mathcal B_n$ with entries
\[
M(n)_{x^\alpha,x^{\alpha'}} := L_\beta(x^\alpha x^{\alpha'}), 
\qquad |\alpha|,|\alpha'|\le n.
\]
We write 
\[
\operatorname{Col} M(n)
:= \{M(n)v : v \in \mathbb{R}^{\dim \mathcal P_n}\}
\subseteq \mathbb{R}^{\dim \mathcal P_n}
\]
and refer to it as the \textbf{column space} of $M(n)$.

For every \(x^\alpha\in\mathcal B_n\), let \(X^\alpha\) denote the column of \(M(n)\) indexed by \(x^\alpha\). If $p(x)=\sum_{|\alpha|\le n}a_\alpha x^\alpha\in\mathcal P_n,$ we define 
$$
    p(X):=\sum_{|\alpha|\le n}a_\alpha X^\alpha \in\operatorname{Col}M(n). 
$$

The moment matrix $M(n)$ is said to be \textbf{recursively generated} if
whenever $p,q,pq \in \mathcal P_n$ and $p(X)=0$ in $\operatorname{Col} M(n)$, then $(pq)(X)=0$ in $\operatorname{Col} M(n)$.

Let \(q\in\mathbb R[x_1,\ldots,x_d]\) be a polynomial of degree
\(\delta\le 2n\). The \textbf{localizing matrix} associated with \(q\),
denoted by \(M_q(n)\equiv M_q(n;\beta)\), is the symmetric matrix
indexed by \(\mathcal B_{n-\left\lceil\frac{\delta}{2}\right\rceil}
\times\mathcal B_{n-\left\lceil\frac{\delta}{2}\right\rceil}\) with entries
\[
M_q(n)_{x^\alpha,x^{\alpha'}}
:=
L_\beta\!\left(qx^{\alpha+\alpha'}\right),
\qquad
|\alpha|,|\alpha'|
\le n-\left\lceil\frac{\delta}{2}\right\rceil.
\]

Let
\[
Q := \{ q_0 := 1, q_1, \dots, q_m \} \subset \mathbb{R}[x_1,\dots,x_d]
\]
and define the semialgebraic set
\[
K_Q := \{ x \in \mathbb{R}^d : q_i(x) \ge 0,\ i = 0,1,\dots,m \}.
\]

Let $k\in \NN_0$. We say that a moment matrix \(M(n+k)\) is a \textbf{positive extension} of \(M(n)\) if \(M(n+k)\succeq0\) and its principal submatrix indexed by \(\mathcal B_n\) coincides with \(M(n)\). Following Curto and Fialkow \cite{CF08}, we say that $Q$ satisfies \textbf{property $(S_{n,k})$} if, for every truncated sequence \(\beta^{(2n)}\), the following equivalence holds:
\begin{equation}
\label{def-S-n-k-property}
\beta^{(2n)} \text{ has a } K_Q\text{--representing measure}
\;\Longleftrightarrow\;
\begin{aligned}
& M(n) \text{ admits a positive extension } M(n+k),\\
& \text{such that } M_{q_i}(n+k) \succeq 0,\quad i=1,\dots,m.
\end{aligned}
\end{equation}
If the right-hand side of \eqref{def-S-n-k-property} holds, we say that $\beta^{(2n)}$ has the \textbf{$k$–extension property relative to $Q$}.\\

\begin{remark}
{
The validity of property \((S_{n,k})\) may depend on the particular choice of generators \(Q=\{1,q_1,\ldots,q_m\}\) defining \(K_Q\), and not only on the semialgebraic set \(K_Q\) itself. We therefore attribute property \((S_{n,k})\) to \(Q\) throughout the paper.
}
\end{remark}

For full positive functionals, the property that vanishing on the principal ideal generated by a defining polynomial depends only on its real zero set is an immediate consequence of the fact that the annihilator ideal of a positive functional is real; see \cite[Lemmas~2.4 and~2.5]{bs16}. More precisely, if \(p,q\in\mathbb R[x_1,\ldots,x_d]\) satisfy \(\cZ_p=\cZ_q\), then a full positive functional vanishes on \((p)\) if and only if it vanishes on \((q)\). The analogous statement in the truncated setting requires additional care. The truncated annihilator is defined only up to a fixed degree and need not be an ideal. We therefore first establish a finite-degree real-radical transfer principle. Its proof uses Real Nullstellensatz certificates and enlarges the truncation order so that all the polynomials occurring in these certificates lie in the domain of the extended Riesz functional.

\begin{lemma}
\label{lem:finite-degree-radical-transfer}
Let \(a,b\in\mathbb R[x_1,\ldots,x_d]\) satisfy
\(a\in I(\cZ_b)\), and let \(r\in\mathbb N_0\). Then there exists
\(N_0\in\mathbb N\) such that, for every \(N\ge N_0\) and every
linear functional \(L:\mathcal P_{2N}\to\mathbb R\) that is
nonnegative on squares and satisfies
\[
        L(bfg)=0
        \qquad
        \text{for all }
        f,g\in
        \mathcal P_{
        N-\left\lceil\frac{\deg b}{2}\right\rceil},
\]
one has 
\[
        L(auv)=0
        \qquad
        \text{for all }
        u,v\in
        \mathcal P_{r}.
\]
\end{lemma}

\begin{proof}
Let \(\mathcal B_r=\{u_1,\ldots,u_m\}\) be the monomial basis of
\(\mathcal P_r\), and set
\[
        \mathcal U
        :=
        \{u_i:1\le i\le m\}
        \cup
        \{u_i+u_j:1\le i<j\le m\}.
\]
For every \(v\in\mathcal U\), the inclusion \(a\in I(\cZ_b)\)
implies \(av^2\in I(\cZ_b)\). By the Real Nullstellensatz
\cite[p.~26]{Mar08}, there exist \(M_v\in\mathbb N\) and polynomials
\(s_{v,1},\ldots,s_{v,m_v}\) such that
\[
        (av^2)^{2M_v}
        +
        \sum_{j=1}^{m_v}s_{v,j}^2
        \in(b).
\]
Choose \(\ell_v\in\mathbb N\) such that
\(2^{\ell_v}\ge2M_v\), and set
\(q_v:=2^{\ell_v-1}-M_v\in\mathbb N_0\). Multiplying the preceding
certificate by \((av^2)^{2q_v}\), and then renaming the resulting
polynomials, gives
\begin{equation}
\label{eq:finite-radical-certificate}
        (av^2)^{2^{\ell_v}}
        +
        \sum_{j=1}^{m_v}s_{v,j}^2
        =
        bh_v.
\end{equation}

Since \(\mathcal U\) is finite, we may choose \(N_0\) sufficiently
large that, for every \(v\in\mathcal U\) and
\(1\le j\le m_v\),
\begin{equation}\label{eq:N0-radical-transfer} 2^{\ell_v-1}\deg(av^2)\le N_0,\qquad \deg s_{v,j}\le N_0,\qquad \deg h_v\le 2\left(N_0-\left\lceil\frac{\deg b}{2}\right\rceil\right). 
\end{equation}

We also require \(\deg(auv)\le2N_0\) for all
\(u,v\in\mathcal P_r\).

Fix \(N\ge N_0\), let \(L\) satisfy the assumptions, and set
\(r_b:=N-\lceil\deg b/2\rceil\). Since every monomial of degree at
most \(2r_b\) is a product of two monomials of degree at most \(r_b\),
linearity gives \(L(bh)=0\) for every \(h\in\mathcal P_{2r_b}\).
In particular, \eqref{eq:N0-radical-transfer} implies
\(L(bh_v)=0\) for every \(v\in\mathcal U\).

Applying \(L\) to \eqref{eq:finite-radical-certificate}, we obtain
\[
        L\bigl((av^2)^{2^{\ell_v}}\bigr)
        +
        \sum_{j=1}^{m_v}L(s_{v,j}^2)
        =
        0.
\]
Every term on the left is nonnegative, and hence
\(L((av^2)^{2^{\ell_v}})=0\).

Positivity of \(L\) gives the Cauchy--Schwarz inequality
\[
        |L(g)|^2\le L(g^2)L(1)
        \qquad
        (g\in\mathcal P_N).
\]
Applying it successively to
\[
        (av^2)^{2^{\ell_v-1}},
        \ (av^2)^{2^{\ell_v-2}},
        \ldots,
        av^2
\]
gives \(L(av^2)=0\) for every \(v\in\mathcal U\). In particular,
\[
        L(au_i^2)=0,
        \qquad
        L\bigl(a(u_i+u_j)^2\bigr)=0
        \quad
        (1\le i<j\le m).
\]
Polarization therefore gives
\[
        2L(au_iu_j)
        =
        L\bigl(a(u_i+u_j)^2\bigr)
        -
        L(au_i^2)
        -
        L(au_j^2)
        =
        0.
\]
Together with the diagonal identities, this shows that
\(L(au_iu_j)=0\) for all \(1\le i,j\le m\). The conclusion
\(L(auv)=0\) for arbitrary \(u,v\in\mathcal P_r\) now follows by
bilinearity.
\end{proof}

The lemma shows that, after passing to a sufficiently high positive
extension, the truncated annihilation of one defining polynomial
implies the truncated annihilation of every polynomial vanishing on
the same algebraic set. We now apply it to property
\((S_{n,k})\).

\begin{proposition}
\label{prop:independence-of-p}
Let \(p,q\in\mathbb R[x_1,\ldots,x_d]\) satisfy
$$
        \cZ_p=\cZ_q,
$$
and let
$
        n\ge\max\{\deg p,\deg q\}.
$
Then the following statements are equivalent:
\begin{enumerate}
\item
\label{S-n-k-pt1}
The generator family \(\{1,p,-p\}\) satisfies property
\((S_{n,k})\) for some \(k\in\mathbb N_0\).

\item
\label{S-n-k-pt2}
The generator family \(\{1,q,-q\}\) satisfies property
\((S_{n,\widetilde k})\) for some
\(\widetilde k\in\mathbb N_0\).
\end{enumerate}
\end{proposition}

\begin{proof}
By symmetry, it suffices to prove
\(\eqref{S-n-k-pt1}\Rightarrow\eqref{S-n-k-pt2}\).
Suppose that \(\{1,p,-p\}\) satisfies property \((S_{n,k})\) for
some \(k\in\mathbb N_0\), and set
\[
        r_p
        :=
        n+k-\left\lceil\frac{\deg p}{2}\right\rceil.
\]
Since \(\cZ_p=\cZ_q\), we have \(p\in I(\cZ_q)\). Apply
Lemma~\ref{lem:finite-degree-radical-transfer} with
\(a=p\), \(b=q\), and \(r=r_p\), and let \(N_0\) be the integer
provided by the lemma.

Choose \(\widetilde k\in\mathbb N_0\) sufficiently large that, with
\(\widetilde n:=n+\widetilde k\),
\[
        \widetilde n\ge\max\{n+k,N_0\}.
\]
We claim that \(\{1,q,-q\}\) satisfies property
\((S_{n,\widetilde k})\).

Let \(\beta^{(2n)}\) be a truncated sequence having the
\(\widetilde k\)-extension property relative to
\(\{1,q,-q\}\). Thus, \(M(n)\) admits a positive extension
\(M(\widetilde n)\) satisfying
\[
        M_q(\widetilde n)\succeq0,
        \qquad
        M_{-q}(\widetilde n)\succeq0.
\]
Since \(M_{-q}(\widetilde n)=-M_q(\widetilde n)\), these inequalities
are equivalent to \(M_q(\widetilde n)=0\).

Let \(L:=L_{\beta^{(2\widetilde n)}}\) be the Riesz functional
associated with this extension. The identity
\(M_q(\widetilde n)=0\) means that
\[
        L(qfg)=0
        \qquad
        \text{for all }
        f,g\in
        \mathcal P_{
        \widetilde n-
        \left\lceil\frac{\deg q}{2}\right\rceil}.
\]
Since \(\widetilde n\ge N_0\), Lemma \ref{lem:finite-degree-radical-transfer} gives
\[
        L(puv)=0
        \qquad
        \text{for all }u,v\in\mathcal P_{r_p}.
\]
Equivalently, \(M_p(n+k)=0\), and hence
\[
        M_p(n+k)\succeq0,
        \qquad
        M_{-p}(n+k)\succeq0.
\]

Since \(\widetilde n\ge n+k\), the truncation of
\(M(\widetilde n)\) to degree \(2(n+k)\) is a positive extension
\(M(n+k)\) of \(M(n)\). Therefore, \(\beta^{(2n)}\) has the
\(k\)-extension property relative to \(\{1,p,-p\}\). Since
\(\{1,p,-p\}\) satisfies property \((S_{n,k})\), the sequence
\(\beta^{(2n)}\) admits a \(\cZ_p\)-representing measure. Because
\(\cZ_p=\cZ_q\), the same measure is also a
\(\cZ_q\)-representing measure.

We have thus proved that the \(\widetilde k\)-extension property
relative to \(\{1,q,-q\}\) implies the existence of a
\(\cZ_q\)-representing measure.

Conversely, suppose that \(\beta^{(2n)}\) has a
\(\cZ_q\)-representing measure. By the Richter--Tchakaloff theorem
\cite{Ric57} (see also \cite[Theorem~1.24]{sch17}),
\(\beta^{(2n)}\) has a finitely atomic representing measure supported
on \(\cZ_q\). Such a measure has moments of every order and therefore
defines positive moment extensions of every order. Since its support
is contained in \(\cZ_q\), the corresponding localizing matrices
satisfy
$
        M_q(\widetilde n)
        =
        M_{-q}(\widetilde n)
        =
        0.
$
Thus, \(\beta^{(2n)}\) has the \(\widetilde k\)-extension property
relative to \(\{1,q,-q\}\).

It follows that \(\{1,q,-q\}\) satisfies property
\((S_{n,\widetilde k})\). This proves
\(\eqref{S-n-k-pt1}\Rightarrow\eqref{S-n-k-pt2}\).
The reverse implication follows by interchanging \(p\) and \(q\).
\end{proof}

{
\begin{remark}
\label{rem:powers-irrelevant}
As an immediate consequence of Proposition \ref{prop:independence-of-p}, for every \(r\in\mathbb N\) and every \(n\ge r\deg p\), the existence of some finite \(k\) such that \(\{1,p,-p\}\) satisfies property \((S_{n,k})\) is equivalent to the existence of some finite \(\widetilde k\) such that \(\{1,p^r,-p^r\}\) satisfies property \((S_{n,\widetilde k})\).
More generally, if
\[
        p=c\,p_1^{m_1}\cdots p_s^{m_s},
        \qquad c\in\mathbb R\setminus\{0\},
\]
where \(p_1,\ldots,p_s\) are irreducible polynomials and
\(m_1,\ldots,m_s\in\mathbb N\), then
$\cZ_p=\cZ_{p_1\cdots p_s}.$
Therefore, when studying whether there exists some finite \(k\) such that
\(\{1,p,-p\}\) satisfies \((S_{n,k})\), it is enough to assume that all
irreducible factors of \(p\) appear with multiplicity one.
Proposition \ref{prop:independence-of-p} preserves only the existence of some finite extension parameter; it does not imply that the same parameter works for \(p\) and its square-free part. Thus, the failure of property \((S_{n,k})\) for every \(k\in\mathbb N_0\) is preserved when passing between \(p\) and its square-free part, whereas a positive result may involve different finite extension parameters.
\end{remark}
}

The \textbf{$n$-th truncated quadratic module} generated by $Q$, denoted by $\Sigma_{Q,n} \subset \mathcal P_{2n}$, is defined as
\begin{equation} \label{def:quadratic-module} 
    \Sigma_{Q,n} := \left\{ \sum_{i=0}^m\sum_{j=1}^{N_i}q_i f_{ij}^2 : N_i\in\mathbb N_0,\ \deg(q_i f_{ij}^2)\le2n \right\}. 
\end{equation}

The following theorem relates the truncated quadratic modules
\(\Sigma_{Q,*}\) to property \((S_{n,k})\).

\begin{theorem}[{\cite[Theorems 1.5, 1.6]{CF08}}]
\label{thm:Psatz-bounds}
With the notation above:
\begin{enumerate}
\item[(i)]
If $Q$ satisfies $(S_{n,k})$, then every polynomial $p\in\mathcal P_{2n}$ that is strictly positive on $K_Q$ belongs to $\Sigma_{Q,n+k}$.

\item[(ii)]
If $k\ge 1$ and every polynomial $p\in\mathcal P_{2n+2}$ that is strictly positive on $K_Q$ belongs to $\Sigma_{Q,n+k}$, then $Q$ satisfies $(S_{n,k})$.

\item[(iii)]
If $K_Q$ is compact and every polynomial $p\in\mathcal P_{2n}$ that is strictly positive on $K_Q$ belongs to $\Sigma_{Q,n}$, then $Q$ satisfies $(S_{n,0})$.
\end{enumerate}
\end{theorem}

Several generator families for which property \((S_{n,k})\) is known to hold are summarized in Table~\ref{tab:known-cases}. In addition, for \(d=1\), the generator family \(Q=\{1\}\), for which \(K_Q=\mathbb R\), satisfies property \((S_{n,1})\) for every \(n\in\NN\) \cite[Th.~3.9]{CF91}.

\begin{table}[ht]
\centering
\caption{Known cases in which \(Q=\{1,p,-p\}\) satisfies property
\((S_{n,k})\), where \(n\ge\deg p\).}
\label{tab:known-cases}
\renewcommand{\arraystretch}{1.15}
\setlength{\tabcolsep}{3pt}
\begin{tabularx}{\textwidth}{
    @{}
    >{\raggedright\arraybackslash}X
    >{\centering\arraybackslash}p{3cm}
    >{\raggedright\arraybackslash}p{4cm}
    @{}
}
\toprule
Polynomial \(p\) & Extension parameter \(k\) & Reference \\
\midrule

\(p(x)=(x-a)(b-x)\), \(a<b\)
&
\(0\)
&
\cite[Th.~II.2.3]{KN77}
\\
\addlinespace[0.35em]

\(p(x,y)=1-x^2-y^2\)
&
\(1\)
&
\cite[Prop.~3.10]{CF08}
\\
\addlinespace[0.35em]

\(p(x,y)=ax+by+c\), \((a,b)\ne(0,0)\)
&
\(1\)
&
\cite[Prop.~3.11]{CF08}
\\
\addlinespace[0.35em]

\(p(x,y)\), \(\deg p=2\), defining an ellipse
&
\(1\)
&
\cite[Prop.~3.13]{CF08}
\\
\addlinespace[0.35em]

\(p(x,y)\), \(\deg p=2\), defining a parabola or a hyperbola
&
\(2\)
&
\cite[Prop.~3.13]{CF08}
\\
\addlinespace[0.35em]

\(p(x,y)=yq(x)-1\), \(\deg q\ge1\)
&
\(k_1\)
&
\cite[Th.~4.1]{f1}
\\
\addlinespace[0.35em]

\(p(x,y)=yx^\ell-1\), \(\ell\ge1\)
&
\(\ell+1\)
&
\cite[Th.~4.1]{z4}
\\
\addlinespace[0.35em]

\(p(x,y)=y-q(x)\), \(\deg q\ge3\)
&
\(\deg q-1\)
&
\cite[Th.~3.1]{z4}
\\

\bottomrule
\end{tabularx}

\medskip
\parbox{\textwidth}{\small
\(k_1:=(2n+2)(2+\deg q)-(n+1+\deg q)\).
}
\end{table}

\section{Generator families \(\{1,p,-p\}\) failing property \((S_{n,k})\)}
\label{sec:nonSnk} 

In this section, we study plane curves for which property \((S_{n,k})\) fails for every finite \(k\). We begin with curves defined by \(x^{\ell_1}-y^{\ell_2}\) and their monomial multiples (Theorems~\ref{theo:case-1}, \ref{theo:case-1.2}, and \ref{theo:case-1.3}), for which we construct explicit sequences exhibiting this failure. As a consequence, we obtain polynomials that are nonnegative on the corresponding curves but are not sums of squares in their coordinate rings; see Corollaries~\ref{pos-not-sos-v2} and \ref{cor-after-3-3-v2}. We then consider the asymmetric families \(p(x,y)=y(y-x^\ell)\) and \(p(x,y)=xy(y-x^\ell)\) (Theorems~\ref{theo:case-6} and \ref{theo:case-7}) and show that the same obstruction persists. The corresponding nonnegative polynomials that are not sums of squares are given in Corollary~\ref{cor-after-3-7-v2}.

\begin{theorem}
\label{theo:case-1}
Let \(\ell_1,\ell_2\in\NN\setminus\{1\}\) be relatively prime, and
let
\[
        p(x,y):=x^{\ell_1}-y^{\ell_2}.
\]
Then, for every \(n\ge\max\{\ell_1,\ell_2\}\), there exists a
truncated sequence \(\beta^{(2n)}\) having the
\(k\)-extension property relative to \(\{1,p,-p\}\) for every
\(k\in\NN\), but admitting no \(\cZ_p\)-representing measure.
Consequently, \(\{1,p,-p\}\) does not satisfy property
\((S_{n,k})\) for any \(k\in\NN_0\).
\end{theorem}

We first establish an auxiliary semigroup-Hankel construction that will be used in the proof of Theorem~\ref{theo:case-1}. It allows us to prescribe a negative moment while preserving the positive definiteness of every finite Hankel restriction indexed by the relevant additive semigroup.

\begin{lemma}
\label{lem:semigroup-Hankel-construction}
Let \(S\subseteq\mathbb N_0\) be a nonempty additive subsemigroup,
and let \(d\in\mathbb N_0\) satisfy \(d/2\notin S\). Then there exists
a sequence
\[
        \gamma
        =
        (\gamma_s)_{s\in(S+S)\cup\{d\}}
\]
such that \(\gamma_d<0\) and, for every finite subset
\(F\subseteq S\), the matrix
\[
        H_\gamma|_F
        :=
        (\gamma_{r+s})_{r,s\in F}
\]
is positive definite.
\end{lemma}

\begin{proof}
Since \(S\subseteq\mathbb N_0\), its elements can be enumerated in
increasing order:
\[
        s_0<s_1<s_2<\cdots.
\]
If \(S\) is finite, the same argument applies and terminates after
finitely many steps.

We prescribe an arbitrary negative value for \(\gamma_d\) and
construct the remaining moments inductively so that
\[
        H_j
        :=
        (\gamma_{s_a+s_b})_{a,b=0}^j
        \succ0
\]
at every stage.

Since \(d/2\notin S\), we have \(2s_0\ne d\). We may therefore
choose \(\gamma_{2s_0}>0\), which gives
\(H_0=(\gamma_{2s_0})\succ0\).

Suppose that \(H_{j-1}\succ0\) has already been constructed. After
adjoining the index \(s_j\), the enlarged matrix has the block form
\[
        H_j
        =
        \begin{pmatrix}
        H_{j-1} & v_j\\
        v_j^{\mathsf T} & \gamma_{2s_j}
        \end{pmatrix},
\]
where
\[
        v_j
        :=
        \bigl(
        \gamma_{s_0+s_j},
        \ldots,
        \gamma_{s_{j-1}+s_j}
        \bigr)^{\mathsf T}.
\]
Some entries of \(v_j\) may already have been prescribed at earlier
stages, and one of them may equal the distinguished moment
\(\gamma_d\). We retain all previously prescribed values and assign arbitrary values, say zero, to the remaining semigroup sums, using the same value whenever two sums coincide.

The diagonal moment \(\gamma_{2s_j}\) has not been prescribed at an
earlier stage. Indeed, every moment occurring in \(H_{j-1}\) has an
index of the form \(s_a+s_b\) with \(a,b<j\), and hence
\[
        s_a+s_b<2s_j.
\]
Moreover, \(2s_j\ne d\), because \(d/2\notin S\). We may therefore
choose \(\gamma_{2s_j}\) sufficiently large that
\[
        \gamma_{2s_j}
        >
        v_j^{\mathsf T}H_{j-1}^{-1}v_j.
\]
By the Schur complement criterion, this implies \(H_j\succ0\).

Proceeding inductively, we obtain a consistently defined sequence
\(\gamma\) on \((S+S)\cup\{d\}\) such that \(H_j\succ0\) at every
stage. Now let \(F\subseteq S\) be finite. Then
\(F\subseteq\{s_0,\ldots,s_j\}\) for some \(j\), so
\(H_\gamma|_F\) is a principal submatrix of \(H_j\). Therefore,
\[
        H_\gamma|_F\succ0.
\]
The prescribed inequality \(\gamma_d<0\) is preserved throughout the
construction.
\end{proof}

\subsection*{A concrete obstruction on the cusp \(x^2=y^3\)}

Before proving Theorem~\ref{theo:case-1}, we illustrate the
construction for \(p(x,y)=x^2-y^3\) and \(n=4\). Thus,
\[
        \cZ_p
        =
        \{(x,y)\in\mathbb R^2:x^2=y^3\}.
\]
We shall construct a truncated bivariate sequence
\[
        \beta^{(8)}
        =
        \{\beta_{i,j}:i,j\in\mathbb N_0,\ i+j\le8\}
\]
that has the \(k\)-extension property relative to
\(\{1,p,-p\}\) for every \(k\in\mathbb N\), but admits no
\(\cZ_p\)-representing measure.

The curve \(\cZ_p\) has the parametrization
\(x=t^3\), \(y=t^2\), where \(t\in\mathbb R\). Hence
\(x^iy^j=t^{3i+2j}\). The associated additive submonoid is
\[
        S
        :=
        \{3i+2j:i,j\in\mathbb N_0\}
        =
        \{0,2,3,4,\ldots\}.
\]
Since \(2\in S\), whereas \(2/2=1\notin S\),
Lemma~\ref{lem:semigroup-Hankel-construction}, applied with \(d=2\),
yields a sequence \(\gamma=(\gamma_s)_{s\in S}\) such that
\begin{equation}
\label{negative-moment}
        \gamma_2<0
        \qquad\text{and}\qquad
        H_\gamma|_F
        :=
        (\gamma_{r+s})_{r,s\in F}
        \succ0
\end{equation}
for every finite subset \(F\subseteq S\).

Define an infinite bivariate sequence by
\[
        \beta_{i,j}
        :=
        \gamma_{3i+2j},
        \qquad
        i,j\in\mathbb N_0.
\]
Its truncation to total degree eight is the sequence
\(\beta^{(8)}\) considered above.

For \(r\in\mathbb N_0\), set
\[
        S_r
        :=
        \{3i+2j:i,j\in\mathbb N_0,\ i+j\le r\}.
\]
In particular,
\[
        S_4
        =
        \{0,2,3,4,5,6,7,8,9,10,11,12\}.
\]
The relation \(X^2=Y^3\) generates, among others, the relations
\(X^3=XY^3\) and \(X^2Y=Y^4\). We may therefore choose
\[
        \mathcal B_4
        =
        \{
        1,Y,X,Y^2,XY,Y^3,XY^2,Y^4,
        XY^3,X^2Y^2,X^3Y,X^4
        \}
\]
as a set containing one representative monomial for each exponent
in \(S_4\).

Under the parametrization \(X=t^3\), \(Y=t^2\), the elements of
\(\mathcal B_4\) correspond, in the displayed order, to
\[
        1,t^2,t^3,t^4,t^5,t^6,t^7,t^8,
        t^9,t^{10},t^{11},t^{12}.
\]
Consequently, the restriction of \(M(4;\beta)\) to the rows and
columns indexed by \(\mathcal B_4\) is equal, up to a simultaneous
permutation of rows and columns, to
\[
        H_\gamma|_{S_4}
        =
        (\gamma_{r+s})_{r,s\in S_4}.
\]
Indeed, if \(X^aY^b,X^cY^d\in\mathcal B_4\), then
\[
        \beta_{a+c,b+d}
        =
        \gamma_{3(a+c)+2(b+d)}
        =
        \gamma_{(3a+2b)+(3c+2d)}.
\]
By \eqref{negative-moment}, \(H_\gamma|_{S_4}\succ0\).

More generally, for every \(r\ge4\), choose one monomial
\(X^iY^j\) of degree at most \(r\) for each exponent in \(S_r\).
The corresponding principal restriction of \(M(r;\beta)\) is,
up to a simultaneous permutation of rows and columns, equal to
\(H_\gamma|_{S_r}\), which is positive definite by
\eqref{negative-moment}. The remaining moments are then assigned consistently so that every monomial column not indexed by \(\mathcal B_r\) coincides with the representative column having the same parametrized exponent. It follows that
\[
        M(r;\beta)\succeq0
        \qquad
        \text{for every }r\ge4.
\]

Moreover, for all \(i,j\in\mathbb N_0\),
\[
        \beta_{i+2,j}
        =
        \gamma_{3(i+2)+2j}
        =
        \gamma_{3i+2j+6}
        =
        \gamma_{3i+2(j+3)}
        =
        \beta_{i,j+3}.
\]
Thus \(X^2=Y^3\) is a recursively generated column relation in every
moment matrix \(M(r;\beta)\). Consequently,
\[
        M_p(r)=M_{-p}(r)=0
\]
whenever these localizing matrices are defined.

Therefore, for every \(k\in\mathbb N\), the matrix
\(M(4+k;\beta)\) is a positive extension of \(M(4;\beta)\) and
satisfies
\[
        M_p(4+k)\succeq0,
        \qquad
        M_{-p}(4+k)\succeq0.
\]
Hence \(\beta^{(8)}\) has the \(k\)-extension property relative to
\(\{1,p,-p\}\) for every \(k\in\mathbb N\).

It remains to show that \(\beta^{(8)}\) admits no
\(\cZ_p\)-representing measure. Suppose, to the contrary, that
\(\mu\) is such a measure. Every point of \(\cZ_p\) has the form
\((t^3,t^2)\), and therefore \(y=t^2\ge0\) on \(\cZ_p\). It follows
that
\[
        \beta_{0,1}
        =
        \int_{\cZ_p}y\,d\mu
        \ge0.
\]
On the other hand, \eqref{negative-moment} gives
\[
        \beta_{0,1}
        =
        \gamma_2
        <0,
\]
which is a contradiction. Thus \(\beta^{(8)}\) admits no
\(\cZ_p\)-representing measure.

This example is the even-odd case of
Theorem~\ref{theo:case-1}, with
\(\ell_1=2\), \(\ell_2=3\), and obstruction index
\(d=\ell_1=2\).

\bigskip

We now apply the preceding semigroup-Hankel construction to prove Theorem~\ref{theo:case-1} for arbitrary relatively prime integers \(\ell_1,\ell_2>1\).

\begin{proof}[Proof of Theorem~\ref{theo:case-1}]
Parametrize \(\cZ_p\) by
\(x=t^{\ell_2}\), \(y=t^{\ell_1}\), where \(t\in\mathbb R\), and set
\[
        S
        :=
        \{i\ell_2+j\ell_1:i,j\in\mathbb N_0\}.
\]
We construct an infinite bivariate sequence
\(\beta=(\beta_{i,j})_{i,j\in\mathbb N_0}\) such that every moment
matrix \(M(r;\beta)\) is positive semidefinite and
\(M_p(r)=M_{-p}(r)=0\), but whose truncation \(\beta^{(2n)}\)
admits no \(\cZ_p\)-representing measure.

We distinguish three cases.

\medskip
\noindent\textbf{Case 1.}
Suppose that \(\ell_1\) is odd and \(\ell_2\) is even, and set
\(d:=\ell_2\). Then \(d\in S\), whereas \(d/2\notin S\).
Indeed, if \(\ell_2<\ell_1\), then
\(0<\ell_2/2<\ell_2<\ell_1\), so \(\ell_2/2\) cannot be a
nonzero element of \(S\). If \(\ell_1<\ell_2\) and
\[
        \frac{\ell_2}{2}
        =
        a\ell_2+b\ell_1
        \qquad
        (a,b\in\mathbb N_0),
\]
then \(a=0\), and hence \(\ell_2=2b\ell_1\). This contradicts
\(\gcd(\ell_1,\ell_2)=1\) and \(\ell_1>1\).

\medskip
\noindent\textbf{Case 2.}
Suppose that \(\ell_1\) is even and \(\ell_2\) is odd, and set
\(d:=\ell_1\). By the same argument, with \(\ell_1\) and
\(\ell_2\) interchanged, \(d\in S\) and \(d/2\notin S\).

\medskip
\noindent\textbf{Case 3.}
Suppose that both \(\ell_1\) and \(\ell_2\) are odd, and set
\(d:=\ell_1+\ell_2\). Clearly, \(d\in S\). We claim that
\(d/2\notin S\). Otherwise, there would exist
\(a,b\in\mathbb N_0\) such that
$
        \frac{\ell_1+\ell_2}{2}
        =
        a\ell_2+b\ell_1,
$
or equivalently,
$
        (2a-1)\ell_2
        =
        (1-2b)\ell_1.
$
If \(a,b\ge1\), the two sides have opposite signs. If \(a=0\), then
\(\ell_2=(2b-1)\ell_1\), contradicting coprimality and
\(\ell_1>1\). Similarly, if \(b=0\), then
\(\ell_1=(2a-1)\ell_2\), which yields the same contradiction.
Therefore, \(d/2\notin S\).

In all three cases, \(d\in S\) and \(d/2\notin S\). Hence Lemma~\ref{lem:semigroup-Hankel-construction} provides a sequence \[ \gamma = (\gamma_s)_{s\in(S+S)\cup\{d\}} \] such that \begin{equation} \label{negative-moment-general} \gamma_d<0 \qquad\text{and}\qquad (\gamma_{u+v})_{u,v\in F}\succ0 \quad \text{for every finite subset }F\subseteq S. \end{equation} Since \(0\in S\), we have \(S+S=S\). Moreover, \(d\in S\), and therefore \[ (S+S)\cup\{d\}=S. \] Thus we may regard \(\gamma\) simply as a sequence \(\gamma=(\gamma_s)_{s\in S}\).

Define
\[
        \beta_{i,j}
        :=
        \gamma_{i\ell_2+j\ell_1},
        \qquad
        i,j\in\mathbb N_0.
\]
For \(r\in\mathbb N_0\), let
\[
        S_r
        :=
        \{i\ell_2+j\ell_1:
        i,j\in\mathbb N_0,\ i+j\le r\}.
\]
Choose one monomial \(X^iY^j\) of degree at most \(r\) for each
element of \(S_r\), and denote the resulting set by
\(\mathcal B_r\). Up to a simultaneous permutation of rows and
columns, the restriction of \(M(r;\beta)\) to
\(\mathcal B_r\) is
\[
        (\gamma_{u+v})_{u,v\in S_r},
\]
which is positive definite by \eqref{negative-moment-general}.

Every monomial of degree at most \(r\) has the same image under the
parametrization as exactly one representative in \(\mathcal B_r\).
Consequently, every column of \(M(r;\beta)\) coincides with one of
the columns indexed by \(\mathcal B_r\). It follows that
\[
        M(r;\beta)\succeq0
        \qquad
        \text{for every }r\in\mathbb N_0.
\]

For all \(i,j\in\mathbb N_0\), we have
\[
        \beta_{i+\ell_1,j}
        =
        \gamma_{i\ell_2+j\ell_1+\ell_1\ell_2}
        =
        \beta_{i,j+\ell_2}.
\]
Thus \(X^{\ell_1}=Y^{\ell_2}\) is a recursively generated column
relation in every moment matrix \(M(r;\beta)\). Consequently,
$
        M_p(r)=M_{-p}(r)=0
$
whenever these localizing matrices are defined.

For every \(k\in\mathbb N\), the matrix \(M(n+k;\beta)\) is
therefore a positive extension of \(M(n;\beta)\) satisfying
\[
        M_p(n+k)\succeq0,
        \qquad
        M_{-p}(n+k)\succeq0.
\]
Hence \(\beta^{(2n)}\) has the \(k\)-extension property relative to
\(\{1,p,-p\}\) for every \(k\in\mathbb N\).

Define \[ f := \begin{cases} x, & \text{in Case~1},\\ y, & \text{in Case~2},\\ xy, & \text{in Case~3}. \end{cases} \] By construction, \(L_\beta(f)=\gamma_d<0\).
Under the parametrization \(x=t^{\ell_2}\), \(y=t^{\ell_1}\), we have \[ f(t^{\ell_2},t^{\ell_1}) = \begin{cases} \bigl(t^{\ell_2/2}\bigr)^2, & \text{if \(\ell_1\) is odd and \(\ell_2\) is even},\\ \bigl(t^{\ell_1/2}\bigr)^2, & \text{if \(\ell_1\) is even and \(\ell_2\) is odd},\\ \bigl(t^{(\ell_1+\ell_2)/2}\bigr)^2, & \text{if \(\ell_1\) and \(\ell_2\) are both odd}. \end{cases} \] Hence \(f\ge0\) on \(\cZ_p\). If \(\beta^{(2n)}\) admitted a \(\cZ_p\)-representing measure \(\mu\), then \[ 0 > L_\beta(f) = \int_{\cZ_p}f\,d\mu \ge0, \] a contradiction. Therefore, \(\beta^{(2n)}\) admits no \(\cZ_p\)-representing measure, although it has the \(k\)-extension property relative to \(\{1,p,-p\}\) for every \(k\in\NN_0\).
\end{proof}

\begin{corollary}
\label{pos-not-sos-v2}
Let \(\ell_1,\ell_2\in\NN\setminus\{1\}\) satisfy
\(\gcd(\ell_1,\ell_2)=1\), and let
\(p(x,y):=x^{\ell_1}-y^{\ell_2}\). Define
\[
f(x,y)
:=
\begin{cases}
x,
& \text{if \(\ell_1\) is odd and \(\ell_2\) is even},\\[0.2em]
y,
& \text{if \(\ell_1\) is even and \(\ell_2\) is odd},\\[0.2em]
xy,
& \text{if \(\ell_1\) and \(\ell_2\) are both odd}.
\end{cases}
\]
Then \(f\) is nonnegative on \(\cZ_p\), but
\[
        f\notin\Sigma_{\{1,p,-p\},n}
        \qquad
        \text{for every }n\in\NN.
\]
Equivalently, the class of \(f\) is not a sum of squares in the
coordinate ring \(\mathbb R[x,y]/(p)\).
\end{corollary}

\begin{proof}
The computation at the end of the proof of Theorem~\ref{theo:case-1}
shows that, under the parametrization \(x=t^{\ell_2}\), \(y=t^{\ell_1}\),
the polynomial \(f(t^{\ell_2},t^{\ell_1})\) is a square in each of the
three parity cases. Hence \(f\ge 0\) on \(\cZ_p\).

Let \(L_\beta\) be the Riesz functional constructed in that proof.
By construction, \(L_\beta(g^2)\ge 0\) and \(L_\beta(ph)=0\) for all
\(g,h\in\mathbb R[x,y]\), while \(L_\beta(f)<0\).

Suppose that the class of \(f\) were a sum of squares in
\(\mathbb R[x,y]/(p)\). Then
\(f=\sum_{i=1}^r g_i^2+ph\) for some
\(g_1,\ldots,g_r,h\in\mathbb R[x,y]\). Applying \(L_\beta\) gives
\[
0>L_\beta(f)=\sum_{i=1}^r L_\beta(g_i^2)\ge 0,
\]
a contradiction. Thus the class of \(f\) is not a sum of squares in
\(\mathbb R[x,y]/(p)\). In particular,
\(f\notin\Sigma_{\{1,p,-p\},n}\) for every \(n\in\mathbb N\), since
any such representation would give a sum-of-squares representation
of the class of \(f\) modulo \((p)\).
\end{proof}

\begin{remark}
\label{rem:additional-non-sos}
The polynomials in Corollary~\ref{pos-not-sos-v2} arise from the
simplest obstruction used in the proof of
Theorem~\ref{theo:case-1}: a moment \(\gamma_d<0\) whose half-index
\(d/2\) does not belong to the additive semigroup
\[
        S
        =
        \{i\ell_2+j\ell_1:i,j\in\mathbb N_0\}.
\]
Although the semigroup-Hankel matrices indexed by finite subsets of
\(S\) are positive definite, \(\gamma_d\) would be a negative
diagonal entry in the full Hankel matrix obtained by adjoining the
monomial \(T^{d/2}\). The corresponding polynomial \(x\), \(y\), or
\(xy\) is nonnegative on \(\cZ_p\), but is evaluated negatively by
\(L_\beta\).

The same separating-functional argument produces further
nonnegative polynomials that are not sums of squares in the
coordinate ring. Let \(a_1,\ldots,a_r\in\mathbb N_0\) be such that
\(a_i+a_j\in S\) for all \(i,j\), and consider
\[
        G
        :=
        (\gamma_{a_i+a_j})_{i,j=1}^r.
\]
Suppose that \(G\) has a negative eigenvalue, and let
\(v=(v_1,\ldots,v_r)^{\mathsf T}\) be a corresponding eigenvector.
If, under the parametrization
\(x=t^{\ell_2}\), \(y=t^{\ell_1}\), the square
\[
        \left(
        v_1t^{a_1}+\cdots+v_rt^{a_r}
        \right)^2
\]
is represented by a polynomial \(f(x,y)\), then \(f\) is
nonnegative on \(\cZ_p\). On the other hand,
\[
        L_\beta(f)
        =
        v^{\mathsf T}Gv
        <0.
\]
The argument from Corollary~\ref{pos-not-sos-v2} therefore shows that
the class of \(f\) is not a sum of squares in
\(\mathbb R[x,y]/(p)\).

For a concrete example, suppose that \(\ell_1\) is odd and
\(\ell_2\) is even, and consider
\[
        G
        :=
        \begin{pmatrix}
        \gamma_{\ell_2}
        &
        \gamma_{\ell_1\ell_2/2}\\
        \gamma_{\ell_1\ell_2/2}
        &
        \gamma_{(\ell_1-1)\ell_2}
        \end{pmatrix}.
\]
Since its first diagonal entry satisfies \(\gamma_{\ell_2}<0\), the
matrix \(G\) is not positive semidefinite. Let
\(v=(v_1,v_2)^{\mathsf T}\) be an eigenvector corresponding to a
negative eigenvalue of \(G\), and define
\[
        f_v(x,y)
        :=
        v_1^2x
        +
        2v_1v_2y^{\ell_2/2}
        +
        v_2^2x^{\ell_1-1}.
\]
Under the parametrization of \(\cZ_p\), this polynomial satisfies
\[
\begin{aligned}
        f_v(t^{\ell_2},t^{\ell_1})
        &=
        v_1^2t^{\ell_2}
        +
        2v_1v_2t^{\ell_1\ell_2/2}
        +
        v_2^2t^{(\ell_1-1)\ell_2}\\
        &=
        \left(
        v_1t^{\ell_2/2}
        +
        v_2t^{(\ell_1-1)\ell_2/2}
        \right)^2.
\end{aligned}
\]
Hence \(f_v\ge0\) on \(\cZ_p\). Nevertheless,
\(L_\beta(f_v)=v^{\mathsf T}Gv<0\). Therefore, the class of \(f_v\)
is not a sum of squares in \(\mathbb R[x,y]/(p)\).
\end{remark}

\medskip

We next show that the obstruction established in Theorem~\ref{theo:case-1} persists after multiplying the defining polynomial by either coordinate variable.

\begin{theorem}\label{theo:case-1.2}
Let $\ell_1,\ell_2\in\NN\setminus\{1\}$ satisfy
$\gcd(\ell_1,\ell_2)=1$, and let
\[
p(x,y)=q(x,y)\bigl(x^{\ell_1}-y^{\ell_2}\bigr),
\qquad
q(x,y)\in\{x,y\}.
\]
Then, for every 
    $n\geq\max\{\ell_1+1,\ell_2+1\},$
there exists a
truncated sequence $\beta^{(2n)}$ having the $k$-extension property
relative to $\{1,p,-p\}$ for every $k\in\NN_0$, but admitting no
$\cZ_p$-representing measure.
\end{theorem}

\begin{proof}
By interchanging $x$ and $y$, it suffices to consider $q(x,y)=y$, so
that
\[
p(x,y)=y\bigl(x^{\ell_1}-y^{\ell_2}\bigr).
\]
Its zero set is
\[
\cZ_p=\cZ_y\cup \cZ_{x^{\ell_1}-y^{\ell_2}}.
\]
The second component is parametrized by $x=t^{\ell_2}$,
$y=t^{\ell_1}$, $t\in\RR$. Consider the additive subsemigroup
\[
S_+:=\{a\ell_2+b\ell_1:a\in\NN_0,\ b\in\NN\}.
\]

We first record an auxiliary construction.

\medskip
\noindent\textbf{Auxiliary lemma.}
Suppose that $d\in S_+$ and $d/2\notin S_+$. Then there exists an
infinite bivariate sequence
$\beta=(\beta_{i,j})_{i,j\in\NN_0}$ such that
\[
M(r;\beta)\succeq0\quad(r\in\NN_0)\qquad\text{and}\qquad
\beta_{i+\ell_1,j+1}=\beta_{i,j+\ell_2+1}
\quad(i,j\in\NN_0),
\]
and $\beta_{a,b}<0$ whenever $a\in\NN_0$, $b\in\NN$, and
$a\ell_2+b\ell_1=d$.

\smallskip
\noindent\textit{Proof of the auxiliary lemma.}
Introduce
\[
\mathcal T
:=
\{\mathsf P_a:a\in\NN_0\}
\,\dot\cup\,
\{\mathsf M_s:s\in S_+\},
\]
with the commutative semigroup operation defined by
$\mathsf P_a+\mathsf P_b:=\mathsf P_{a+b}$,
$\mathsf P_a+\mathsf M_s:=\mathsf M_{a\ell_2+s}$, and
$\mathsf M_s+\mathsf M_t:=\mathsf M_{s+t}$.
Thus $\mathsf P_a$ represents the pure monomial $X^a$, whereas
$\mathsf M_s$ represents the class of monomials $X^aY^b$ with
$b\geq1$ and $a\ell_2+b\ell_1=s$. In particular, pure and mixed
classes are kept distinct.

Define
$w(\mathsf P_a):=a\ell_2$ and $w(\mathsf M_s):=s$.
Enumerate the elements of $\mathcal T$ as
$\tau_0,\tau_1,\tau_2,\ldots$ in nondecreasing order of weight, with
the mixed element placed before the pure element whenever the two
have the same weight.

We construct $\varphi:\mathcal T\to\RR$ such that
$\varphi(\mathsf M_d)<0$ and
\[
\bigl(\varphi(\sigma+\tau)\bigr)_{\sigma,\tau\in F}\succ0
\]
for every finite $F\subseteq\mathcal T$.

Prescribe an arbitrary negative value for $\varphi(\mathsf M_d)$.
Suppose inductively that
\[
H_{j-1}
:=
\bigl(\varphi(\tau_a+\tau_b)\bigr)_{a,b=0}^{j-1}
\succ0
\]
has already been constructed. After adjoining $\tau_j$, write
\[
H_j=
\begin{pmatrix}
H_{j-1} & v_j\\
v_j^T & \varphi(2\tau_j)
\end{pmatrix}.
\]
Retain the values in $v_j$ that have already been prescribed, and
assign arbitrary values, say zero, to the remaining entries.

The diagonal value $\varphi(2\tau_j)$ has not been prescribed at an
earlier stage. Indeed, if
$2\tau_j=\tau_a+\tau_b$ for some $a,b<j$, then equality of weights
forces $w(\tau_a)=w(\tau_b)=w(\tau_j)$. If $\tau_j$ is mixed, this is
impossible because the mixed element is placed first at each weight.
If $\tau_j$ is pure, the only possible earlier element of the same
weight is mixed, but the sum of two mixed elements is mixed whereas
$2\tau_j$ is pure. Moreover, $2\tau_j\neq\mathsf M_d$: this is clear
when $\tau_j$ is pure, while in the mixed case equality would imply
$d/2\in S_+$.

Hence $\varphi(2\tau_j)$ is free, and we choose it so that
\[
\varphi(2\tau_j)>v_j^TH_{j-1}^{-1}v_j.
\]
The Schur complement criterion gives $H_j\succ0$. Proceeding
inductively yields the desired function $\varphi$.

Now define
\[
\beta_{i,0}:=\varphi(\mathsf P_i),\qquad
\beta_{i,j}:=\varphi\bigl(\mathsf M_{i\ell_2+j\ell_1}\bigr)
\quad(j\geq1).
\]
Associate to each monomial $X^iY^j$ the label
\[
\lambda(i,j):=
\begin{cases}
\mathsf P_i, & j=0,\\
\mathsf M_{i\ell_2+j\ell_1}, & j\geq1.
\end{cases}
\]
Then $\lambda(i,j)+\lambda(a,b)=\lambda(i+a,j+b)$, and therefore
\[
\beta_{i+a,j+b}
=
\varphi\bigl(\lambda(i,j)+\lambda(a,b)\bigr).
\]
Thus every moment matrix $M(r;\beta)$ is the pullback of a finite
positive definite matrix
$$\bigl(\varphi(\sigma+\tau)\bigr)_{\sigma,\tau\in F},$$ and hence
\[
M(r;\beta)\succeq0\qquad(r\in\NN_0).
\]

Furthermore, for all $i,j\in\NN_0$,
\[
\beta_{i+\ell_1,j+1}
=
\beta_{i,j+\ell_2+1},
\]
because the two moments correspond to the same mixed label. Finally,
if $d=a\ell_2+b\ell_1$ with $b\geq1$, then
$\beta_{a,b}=\varphi(\mathsf M_d)<0$. This proves the auxiliary
lemma.\hfill $\blacksquare$

\medskip
We now choose the obstruction index $d$.
Suppose first that $\ell_1$ is odd and $\ell_2$ is even, and set
\[
d:=\ell_2+2\ell_1.
\]
Then $d\in S_+$. If $d/2\in S_+$, then
$\ell_2/2+\ell_1=a\ell_2+b\ell_1$ for some
$a\in\NN_0$, $b\in\NN$, and hence
$(2a-1)\ell_2=2(1-b)\ell_1$. If $a\geq1$, the two sides have
opposite signs; hence $a=0$, which gives
$\ell_2=2(b-1)\ell_1$, contradicting coprimality and $\ell_1>1$.
Thus $d/2\notin S_+$.

Suppose next that $\ell_1$ is even and $\ell_2$ is odd, and set
$d:=\ell_1$. Then $d\in S_+$ and $d/2\notin S_+$, since every
element of $S_+$ is at least $\ell_1$.

Finally, suppose that both $\ell_1$ and $\ell_2$ are odd, and set
$d:=\ell_1+\ell_2$. If $d/2\in S_+$, then
$(\ell_1+\ell_2)/2=a\ell_2+b\ell_1$ for some
$a\in\NN_0$, $b\in\NN$, and hence
$(2a-1)\ell_2=(1-2b)\ell_1$. Since $b\geq1$, the right-hand side is
negative, so $a=0$, which gives
$\ell_2=(2b-1)\ell_1$, again contradicting coprimality and
$\ell_1>1$. Thus $d/2\notin S_+$.

In each case, the auxiliary lemma gives an infinite sequence
$\beta=(\beta_{i,j})_{i,j\in\NN_0}$ satisfying
\[
M(r;\beta)\succeq0\qquad(r\in\NN_0)
\]
and
\[
\beta_{i+\ell_1,j+1}
=
\beta_{i,j+\ell_2+1}
\qquad(i,j\in\NN_0).
\]
Hence
\[
Y(X^{\ell_1}-Y^{\ell_2})=0
\]
is a recursively generated column relation at every order, and
therefore
\[
M_p(r)=M_{-p}(r)=0
\]
whenever the localizing matrices are defined.

Define
\[
f:=
\begin{cases}
xy^2,
& \ell_1\ \text{odd},\ \ell_2\ \text{even},\\
y,
& \ell_1\ \text{even},\ \ell_2\ \text{odd},\\
xy,
& \ell_1,\ell_2\ \text{odd}.
\end{cases}
\]
By construction, $L_\beta(f)<0$. Under the parametrization
$x=t^{\ell_2}$, $y=t^{\ell_1}$,
\[
f(t^{\ell_2},t^{\ell_1})
=
\begin{cases}
\bigl(t^{\ell_2/2+\ell_1}\bigr)^2,
& \ell_1\ \text{odd},\ \ell_2\ \text{even},\\
\bigl(t^{\ell_1/2}\bigr)^2,
& \ell_1\ \text{even},\ \ell_2\ \text{odd},\\
\bigl(t^{(\ell_1+\ell_2)/2}\bigr)^2,
& \ell_1,\ell_2\ \text{odd}.
\end{cases}
\]
Moreover, $f$ vanishes on $\cZ_y$. Hence $f\geq0$ on $\cZ_p$, whereas
$L_\beta(f)<0$.

Let $\beta^{(2n)}$ be the degree-$2n$ truncation of $\beta$. Since
all moment matrices are positive semidefinite and
$M_p(r)=M_{-p}(r)=0$ at every order, $M(n+k;\beta)$ is a positive
extension of $M(n;\beta)$ satisfying the required localizing
conditions for every $k\in\NN_0$. Thus $\beta^{(2n)}$ has the
$k$-extension property relative to $\{1,p,-p\}$ for every
$k\in\NN_0$.

If $\beta^{(2n)}$ admitted a $\cZ_p$-representing measure $\mu$, then
\[
0>L_\beta(f)=\int_{\cZ_p}f\,d\mu\geq0,
\]
a contradiction. Therefore, $\beta^{(2n)}$ admits no
$\cZ_p$-representing measure.
\end{proof}

We next show that the same obstruction persists when both coordinate
axes are adjoined to the monomial curve.

\begin{theorem}\label{theo:case-1.3}
Let $\ell_1,\ell_2\in\NN\setminus\{1\}$ satisfy
$\gcd(\ell_1,\ell_2)=1$, and let
\[
p(x,y):=xy\bigl(x^{\ell_1}-y^{\ell_2}\bigr).
\]
Then, for every
$
n\geq\max\{\ell_1+2,\ell_2+2\},
$
there exists a truncated sequence $\beta^{(2n)}$ having the
$k$-extension property relative to $\{1,p,-p\}$ for every
$k\in\NN_0$, but admitting no $\cZ_p$-representing measure.
\end{theorem}

\begin{proof}
The zero set of $p$ is
\[
\cZ_p=\cZ_x\cup \cZ_y\cup \cZ_{x^{\ell_1}-y^{\ell_2}}.
\]
The last component is parametrized by $x=t^{\ell_2}$,
$y=t^{\ell_1}$, $t\in\RR$. Consider the additive subsemigroup
\[
S_{++}:=\{a\ell_2+b\ell_1:a,b\in\NN\}.
\]
Its elements are precisely the exponents associated with monomials
having positive powers of both variables.

We first record an auxiliary construction.

\medskip
\noindent\textbf{Auxiliary lemma.}
Suppose that $d\in S_{++}$ and $d/2\notin S_{++}$. Then there exists
an infinite bivariate sequence
$\beta=(\beta_{i,j})_{i,j\in\NN_0}$ such that
\[
M(r;\beta)\succeq0\quad(r\in\NN_0)
\qquad \text{and}\qquad
\beta_{i+\ell_1+1,j+1}
=
\beta_{i+1,j+\ell_2+1}
\quad(i,j\in\NN_0),
\]
and $\beta_{a,b}<0$ whenever $a,b\in\NN$ and
$a\ell_2+b\ell_1=d$.

\smallskip
\noindent\textit{Proof of the auxiliary lemma.}
Introduce the set
\[
\mathcal T
:=
\{\mathsf e\}
\,\dot\cup\,
\{\mathsf X_a:a\in\NN\}
\,\dot\cup\,
\{\mathsf Y_b:b\in\NN\}
\,\dot\cup\,
\{\mathsf M_s:s\in S_{++}\}.
\]
We make $\mathcal T$ into a commutative semigroup with identity
$\mathsf e$ by setting
\[
\begin{aligned}
\mathsf X_a+\mathsf X_c&:=\mathsf X_{a+c},&
\mathsf Y_b+\mathsf Y_e&:=\mathsf Y_{b+e},\\
\mathsf X_a+\mathsf Y_b&:=\mathsf M_{a\ell_2+b\ell_1},&
\mathsf X_a+\mathsf M_s&:=\mathsf M_{a\ell_2+s},\\
\mathsf Y_b+\mathsf M_s&:=\mathsf M_{b\ell_1+s},&
\mathsf M_s+\mathsf M_t&:=\mathsf M_{s+t}.
\end{aligned}
\]
Thus $\mathsf X_a$ and $\mathsf Y_b$ represent the pure monomials
$X^a$ and $Y^b$, whereas $\mathsf M_s$ represents the class of
mixed monomials $X^aY^b$, $a,b\geq1$, having parametrized exponent
$a\ell_2+b\ell_1=s$. In particular, pure and mixed classes are kept
distinct.

Define a weight on $\mathcal T$ by
\[
w(\mathsf e):=0,\qquad
w(\mathsf X_a):=a\ell_2,\qquad
w(\mathsf Y_b):=b\ell_1,\qquad
w(\mathsf M_s):=s.
\]
Enumerate $\mathcal T$ as
$\tau_0,\tau_1,\tau_2,\ldots$ in nondecreasing order of weight,
with $\tau_0=\mathsf e$, and with the mixed element placed before
the pure elements whenever several elements have the same weight.

We construct a function $\varphi:\mathcal T\to\RR$ such that
$\varphi(\mathsf M_d)<0$ and
\[
\bigl(\varphi(\sigma+\tau)\bigr)_{\sigma,\tau\in F}\succ0
\]
for every finite subset $F\subseteq\mathcal T$.

Prescribe an arbitrary negative value for $\varphi(\mathsf M_d)$.
Suppose inductively that
\[
H_{j-1}
:=
\bigl(\varphi(\tau_a+\tau_b)\bigr)_{a,b=0}^{j-1}
\succ0
\]
has already been constructed. After adjoining $\tau_j$, write
\[
H_j=
\begin{pmatrix}
H_{j-1}&v_j\\
v_j^T&\varphi(2\tau_j)
\end{pmatrix}.
\]
We retain all values in $v_j$ that have already been prescribed and
assign arbitrary values, say zero, to the remaining semigroup sums,
using the same value whenever two sums coincide.

The diagonal value $\varphi(2\tau_j)$ is still free. Indeed, if
$2\tau_j$ coincided with a sum of two elements already present at
this stage, other than $\tau_j+\tau_j$, equality of weights would
force both summands to have the same weight as $\tau_j$. If
$\tau_j=\mathsf M_s$ is mixed, this is impossible because the mixed
element is placed first among the elements of weight $s$. If
$\tau_j=\mathsf X_a$ is pure, only the sum of two pure $X$-elements
can be pure of type $\mathsf X$, and there is only one pure
$X$-element of a given weight. The same argument applies to a pure
$Y$-element. Finally, if $\tau_j=\mathsf e$, its diagonal is
$\mathsf e$ itself and is chosen positive at the initial step.

Moreover, $2\tau_j\neq\mathsf M_d$. This is immediate when
$\tau_j$ is pure or $\tau_j=\mathsf e$. If
$\tau_j=\mathsf M_s$, equality would imply $2s=d$, contrary to
$d/2\notin S_{++}$.

Hence $\varphi(2\tau_j)$ can be chosen sufficiently large that
\[
\varphi(2\tau_j)>v_j^TH_{j-1}^{-1}v_j.
\]
The Schur complement criterion gives $H_j\succ0$. Proceeding
inductively yields $\varphi$ such that every finite matrix
$\bigl(\varphi(\sigma+\tau)\bigr)_{\sigma,\tau\in F}$ is positive
definite.

Associate to each monomial $X^iY^j$ the label
\[
\lambda(i,j):=
\begin{cases}
\mathsf e, & i=j=0,\\
\mathsf X_i, & i\geq1,\ j=0,\\
\mathsf Y_j, & i=0,\ j\geq1,\\
\mathsf M_{i\ell_2+j\ell_1}, & i,j\geq1.
\end{cases}
\]
Then
\[
\lambda(i,j)+\lambda(a,b)=\lambda(i+a,j+b)
\]
for all $i,j,a,b\in\NN_0$. Define
\[
\beta_{i,j}:=\varphi\bigl(\lambda(i,j)\bigr).
\]
It follows that
\[
\beta_{i+a,j+b}
=
\varphi\bigl(\lambda(i,j)+\lambda(a,b)\bigr).
\]
Thus every moment matrix $M(r;\beta)$ is the pullback of a finite
positive definite matrix
$$\bigl(\varphi(\sigma+\tau)\bigr)_{\sigma,\tau\in F},$$ and hence
\[
M(r;\beta)\succeq0\qquad(r\in\NN_0).
\]

Furthermore, for all $i,j\in\NN_0$,
\[
\beta_{i+\ell_1+1,j+1}
=
\beta_{i+1,j+\ell_2+1},
\]
since both moments correspond to the mixed label with exponent
\[
(i+1)\ell_2+(j+1)\ell_1+\ell_1\ell_2.
\]
Finally, if $d=a\ell_2+b\ell_1$ with $a,b\in\NN$, then
$\beta_{a,b}=\varphi(\mathsf M_d)<0$. This proves the auxiliary
lemma.\hfill $\blacksquare$

\medskip
We now choose the obstruction index $d$, distinguishing three parity
cases.

Suppose first that $\ell_1$ is odd and $\ell_2$ is even, and set
\[
d:=\ell_2+2\ell_1.
\]
Then $d\in S_{++}$, corresponding to the moment $\beta_{1,2}$, while
\[
\frac d2=\frac{\ell_2}{2}+\ell_1<\ell_2+\ell_1.
\]
Since every element of $S_{++}$ is at least $\ell_1+\ell_2$, it
follows that $d/2\notin S_{++}$.

Suppose next that $\ell_1$ is even and $\ell_2$ is odd, and set
\[
d:=2\ell_2+\ell_1.
\]
Then $d\in S_{++}$, corresponding to $\beta_{2,1}$, whereas
$d/2=\ell_2+\ell_1/2<\ell_1+\ell_2$. Hence
$d/2\notin S_{++}$.

Finally, suppose that both $\ell_1$ and $\ell_2$ are odd, and set
\[
d:=\ell_1+\ell_2.
\]
Then $d\in S_{++}$, corresponding to $\beta_{1,1}$, and
$d/2<\ell_1+\ell_2$, so again $d/2\notin S_{++}$.

In each case, the auxiliary lemma gives a single infinite sequence
$\beta=(\beta_{i,j})_{i,j\in\NN_0}$ satisfying
\[
M(r;\beta)\succeq0\quad(r\in\NN_0)
\qquad
\text{and}
\qquad
\beta_{i+\ell_1+1,j+1}
=
\beta_{i+1,j+\ell_2+1}
\quad(i,j\in\NN_0).
\]
Thus
\[
XY(X^{\ell_1}-Y^{\ell_2})=0
\]
is a recursively generated column relation at every order.
Consequently,
\[
M_p(r)=M_{-p}(r)=0
\]
whenever the localizing matrices are defined.

Define
\[
f:=
\begin{cases}
xy^2,
& \ell_1\ \text{odd},\ \ell_2\ \text{even},\\
x^2y,
& \ell_1\ \text{even},\ \ell_2\ \text{odd},\\
xy,
& \ell_1,\ell_2\ \text{odd}.
\end{cases}
\]
By construction, $L_\beta(f)<0$. On the component
$\cZ_{x^{\ell_1}-y^{\ell_2}}$, parametrized by
$x=t^{\ell_2}$ and $y=t^{\ell_1}$, we have
\[
f(t^{\ell_2},t^{\ell_1})
=
\begin{cases}
\bigl(t^{\ell_2/2+\ell_1}\bigr)^2,
& \ell_1\ \text{odd},\ \ell_2\ \text{even},\\
\bigl(t^{\ell_2+\ell_1/2}\bigr)^2,
& \ell_1\ \text{even},\ \ell_2\ \text{odd},\\
\bigl(t^{(\ell_1+\ell_2)/2}\bigr)^2,
& \ell_1,\ell_2\ \text{odd}.
\end{cases}
\]
Moreover, $f$ vanishes on both coordinate axes. Hence
$f\geq0$ on $\cZ_p$, whereas $L_\beta(f)<0$.

Let $\beta^{(2n)}$ be the degree-$2n$ truncation of $\beta$. Since
all moment matrices are positive semidefinite and
$M_p(r)=M_{-p}(r)=0$ at every order, $M(n+k;\beta)$ is a positive
extension of $M(n;\beta)$ satisfying the required localizing
conditions for every $k\in\NN_0$. Thus $\beta^{(2n)}$ has the
$k$-extension property relative to $\{1,p,-p\}$ for every
$k\in\NN_0$.

If $\beta^{(2n)}$ admitted a $\cZ_p$-representing measure $\mu$, then
\[
0>L_\beta(f)=\int_{\cZ_p}f\,d\mu\geq0,
\]
a contradiction. Therefore, $\beta^{(2n)}$ admits no
$\cZ_p$-representing measure.
\end{proof}

The preceding constructions also yield explicit nonnegative polynomials
that fail to be sums of squares in the corresponding coordinate rings.

\begin{corollary}
\label{cor-after-3-3-v2}
Let \(\ell_1,\ell_2\in\NN\setminus\{1\}\) satisfy
\(\gcd(\ell_1,\ell_2)=1\), and let
\[
        p(x,y)
        =
        q(x,y)\bigl(x^{\ell_1}-y^{\ell_2}\bigr),
        \qquad
        q(x,y)\in\{x,y,xy\}.
\]
Define
\[
f(x,y)
:=
\begin{cases}
x,
& \text{if \(q=x\), \(\ell_1\) is odd, and \(\ell_2\) is even},\\[0.2em]
xy^2,
& \text{if \(q\in\{y,xy\}\), \(\ell_1\) is odd, and
\(\ell_2\) is even},\\[0.2em]
y,
& \text{if \(q=y\), \(\ell_1\) is even, and \(\ell_2\) is odd},\\[0.2em]
x^2y,
& \text{if \(q\in\{x,xy\}\), \(\ell_1\) is even, and
\(\ell_2\) is odd},\\[0.2em]
xy,
& \text{if \(\ell_1\) and \(\ell_2\) are both odd}.
\end{cases}
\]
Then \(f\) is nonnegative on \(\cZ_p\), but
\[
        f\notin\Sigma_{\{1,p,-p\},n}
        \qquad
        \text{for every }n\in\NN.
\]
Equivalently, the class of \(f\) is not a sum of squares in the coordinate ring \(\mathbb R[x,y]/(p)\).
\end{corollary}

The proof is identical to that of Corollary~\ref{pos-not-sos-v2}, using the separating Riesz functionals and the corresponding obstruction moments constructed in Theorems~\ref{theo:case-1.2} and~\ref{theo:case-1.3}.\\

We next consider the reducible curve obtained by adjoining the line \(\cZ_y\) to \(\cZ_{y-x^\ell}\), and show that the obstruction to property \((S_{n,k})\) persists.

\begin{theorem}
\label{theo:case-6}
Let \(p(x,y):=y(y-x^\ell)\), where
\(\ell\in\NN\setminus\{1\}\). Then, for every \(n\ge\ell+1\),
there exists a truncated sequence \(\beta^{(2n)}\) having the
\(k\)-extension property relative to \(\{1,p,-p\}\) for every
\(k\in\NN_0\), but admitting no \(\cZ_p\)-representing measure.
\end{theorem}

\begin{proof}
Since \(y(y-x^\ell)=-y(x^\ell-y)\), the generator families
associated with \(y(y-x^\ell)\) and \(y(x^\ell-y)\) coincide, i.e.,
        $\{1,p,-p\}
        =
        \{1,y(x^\ell-y),-y(x^\ell-y)\}.$
Although Theorem~\ref{theo:case-1.2} is stated under the assumption
that both exponents are greater than one, its proof remains valid in
the boundary case \(\ell_1=\ell\), \(\ell_2=1\). Indeed, the component
\(\cZ_{x^\ell-y}\) is parametrized by \(x=t\), \(y=t^\ell\), where
\(t\in\mathbb R\), and the corresponding mixed exponents form the
additive subsemigroup
\(S_+=\{a+b\ell:a\in\NN_0,\ b\in\NN\}\).

Suppose first that \(\ell\) is even. Set \(d:=\ell\). Then
\(d\in S_+\), whereas \(d/2\notin S_+\), since every element of
\(S_+\) is at least \(\ell\). The auxiliary construction in the proof
of Theorem~\ref{theo:case-1.2} therefore gives an infinite sequence
\(\beta=(\beta_{i,j})_{i,j\in\NN_0}\) such that all moment matrices
are positive semidefinite,
\(\beta_{i+\ell,j+1}=\beta_{i,j+2}\) for all \(i,j\in\NN_0\), and
\(\beta_{0,1}=\varphi(\mathsf M_\ell)<0\). Under the parametrization,
\(y=t^\ell=(t^{\ell/2})^2\), while \(y=0\) on \(\cZ_y\). Hence
\(y\ge0\) on \(\cZ_p\), whereas
\(L_\beta(y)=\beta_{0,1}<0\).

Suppose now that \(\ell\) is odd. Set \(d:=\ell+1\). Then
\(d\in S_+\), while \(d/2\notin S_+\), since
\((\ell+1)/2<\ell\) and every element of \(S_+\) is at least
\(\ell\). The same auxiliary construction gives an infinite sequence
\(\beta=(\beta_{i,j})_{i,j\in\NN_0}\) such that all moment matrices
are positive semidefinite,
\(\beta_{i+\ell,j+1}=\beta_{i,j+2}\) for all \(i,j\in\NN_0\), and
\(\beta_{1,1}=\varphi(\mathsf M_{\ell+1})<0\). Under the
parametrization,
\(xy=t^{\ell+1}=(t^{(\ell+1)/2})^2\), while \(xy=0\) on
\(\cZ_y\). Thus \(xy\ge0\) on \(\cZ_p\), whereas
\(L_\beta(xy)=\beta_{1,1}<0\).

In either parity case,
\(\beta_{i+\ell,j+1}=\beta_{i,j+2}\) for all \(i,j\in\NN_0\).
Thus \(Y(X^\ell-Y)=0\) is a recursively generated column relation at
every order. Since
\(Y(Y-X^\ell)=-Y(X^\ell-Y)\), we have
\[
        M_p(r)=M_{-p}(r)=0
\]
whenever the localizing matrices are defined. Consequently,
\(\beta^{(2n)}\) has the \(k\)-extension property relative to
\(\{1,p,-p\}\) for every \(k\in\NN_0\).

In either parity case, the constructed Riesz functional takes a
negative value on a polynomial that is nonnegative on \(\cZ_p\).
Such a functional cannot arise from a \(\cZ_p\)-representing measure.
Therefore, \(\beta^{(2n)}\) admits no
\(\cZ_p\)-representing measure.
\end{proof}

\begin{remark}
\label{rem:asymmetry-y-factor}
The preceding construction is asymmetric and does not apply directly
to \(y(x-y^\ell)\). Indeed, the component
\(\cZ_{x-y^\ell}\) is parametrized by \(x=t^\ell\), \(y=t\), so the
mixed exponents are of the form \(a\ell+b\), where
\(a\in\NN_0\) and \(b\in\NN\). Thus the corresponding exponent set
contains every positive integer. Consequently, the half-indices used
to create the negative diagonal obstructions in the proof of
Theorem~\ref{theo:case-6} already belong to the relevant exponent
set, and the same construction cannot be applied.

This reflects a genuine difference between the two families:
\(y(x-y^\ell)\) satisfies property \((S_{n,k})\) for a finite
extension parameter \(k\), as shown in
Theorem~\ref{theo:case-5}.
\end{remark}

We conclude this family of examples by showing that the same obstruction persists when both coordinate axes are adjoined to the curve \(\cZ_{y-x^\ell}\).

\begin{theorem}
\label{theo:case-7}
Let \(p(x,y):=xy(y-x^\ell)\), where
\(\ell\in\NN\setminus\{1\}\). Then, for every \(n\ge\ell+2\),
there exists a truncated sequence \(\beta^{(2n)}\) having the
\(k\)-extension property relative to \(\{1,p,-p\}\) for every
\(k\in\NN_0\), but admitting no \(\cZ_p\)-representing measure.
\end{theorem}

\begin{proof}
Since \(xy(y-x^\ell)=-xy(x^\ell-y)\), we have
\[
        \{1,p,-p\}
        =
        \{1,xy(x^\ell-y),-xy(x^\ell-y)\}.
\]
Although Theorem~\ref{theo:case-1.3} is stated under the assumption
that both exponents are greater than one, its auxiliary construction
remains valid in the boundary case \(\ell_1=\ell\), \(\ell_2=1\).
Indeed, the component \(\cZ_{x^\ell-y}\) is parametrized by
\(x=t\), \(y=t^\ell\), where \(t\in\mathbb R\), and the relevant
mixed exponent semigroup becomes
\(S_{++}=\{a+b\ell:a,b\in\NN\}\).

Suppose first that \(\ell\) is even. Set \(d:=\ell+2\). Then
\(d\in S_{++}\), corresponding to the moment \(\beta_{2,1}\), while
\(d/2=\ell/2+1<\ell+1\). Since every element of \(S_{++}\) is at
least \(\ell+1\), we have \(d/2\notin S_{++}\). The auxiliary
construction in the proof of Theorem~\ref{theo:case-1.3} therefore
yields an infinite sequence \(\beta=(\beta_{i,j})_{i,j\in\NN_0}\)
such that all moment matrices are positive semidefinite,
\(\beta_{i+\ell+1,j+1}=\beta_{i+1,j+2}\) for all
\(i,j\in\NN_0\), and
\(\beta_{2,1}=\varphi(\mathsf M_{\ell+2})<0\). Under the
parametrization,
\(x^2y=t^{\ell+2}=(t^{\ell/2+1})^2\), while \(x^2y\) vanishes on
both coordinate axes. Hence \(x^2y\ge0\) on \(\cZ_p\), whereas
\(L_\beta(x^2y)=\beta_{2,1}<0\).

Suppose now that \(\ell\) is odd. Set \(d:=\ell+1\). Then
\(d\in S_{++}\), corresponding to the moment \(\beta_{1,1}\), while
\(d/2=(\ell+1)/2<\ell+1\). Hence \(d/2\notin S_{++}\). The same
auxiliary construction gives an infinite sequence
\(\beta=(\beta_{i,j})_{i,j\in\NN_0}\) such that all moment matrices
are positive semidefinite,
\(\beta_{i+\ell+1,j+1}=\beta_{i+1,j+2}\) for all
\(i,j\in\NN_0\), and
\(\beta_{1,1}=\varphi(\mathsf M_{\ell+1})<0\). Under the
parametrization,
\(xy=t^{\ell+1}=(t^{(\ell+1)/2})^2\), while \(xy\) vanishes on both
coordinate axes. Thus \(xy\ge0\) on \(\cZ_p\), whereas
\(L_\beta(xy)=\beta_{1,1}<0\).

In either parity case,
\(\beta_{i+\ell+1,j+1}=\beta_{i+1,j+2}\) for all
\(i,j\in\NN_0\). Thus \(XY(X^\ell-Y)=0\) is a recursively generated
column relation at every order. Since
\(XY(Y-X^\ell)=-XY(X^\ell-Y)\), we have
\[
        M_p(r)=M_{-p}(r)=0
\]
whenever the localizing matrices are defined. Consequently,
\(\beta^{(2n)}\) has the \(k\)-extension property relative to
\(\{1,p,-p\}\) for every \(k\in\NN_0\).

According to the parity of \(\ell\), the constructed Riesz functional
takes a negative value on \(x^2y\) or \(xy\), although the
corresponding polynomial is nonnegative on \(\cZ_p\). Hence it cannot
arise from a \(\cZ_p\)-representing measure. Therefore,
\(\beta^{(2n)}\) admits no \(\cZ_p\)-representing measure.
\end{proof}

\begin{corollary}
\label{cor-after-3-7-v2}
Let \(\ell\in\NN\setminus\{1\}\), and let
\[
        p(x,y)
        \in
        \{
        y(y-x^\ell),
        \ x(x-y^\ell),
        \ xy(y-x^\ell),
        \ xy(x-y^\ell)
        \}.
\]
Define
\[
f(x,y)
:=
\begin{cases}
y,
& \text{if \(p(x,y)=y(y-x^\ell)\) and \(\ell\) is even},\\[0.2em]
x,
& \text{if \(p(x,y)=x(x-y^\ell)\) and \(\ell\) is even},\\[0.2em]
x^2y,
& \text{if \(p(x,y)=xy(y-x^\ell)\) and \(\ell\) is even},\\[0.2em]
xy^2,
& \text{if \(p(x,y)=xy(x-y^\ell)\) and \(\ell\) is even},\\[0.2em]
xy,
& \text{if \(\ell\) is odd}.
\end{cases}
\]
Then \(f\) is nonnegative on \(\cZ_p\), but
\[
        f\notin\Sigma_{\{1,p,-p\},n}
        \qquad
        \text{for every }n\in\NN.
\]
Equivalently, \(f\) is not a sum of squares in the coordinate ring
\(\mathbb R[x,y]/(p)\).
\end{corollary}

The proof is identical to that of Corollary~\ref{pos-not-sos-v2}, using the separating Riesz functionals and obstruction moments constructed in Theorems~\ref{theo:case-6} and~\ref{theo:case-7}, together with the symmetric cases obtained by interchanging \(x\) and \(y\).

\section{Generator families \(\{1,p,-p\}\) satisfying property \((S_{n,k})\)}
\label{sec:Snk}

In this section, we establish property \((S_{n,k})\) for several families of planar algebraic curves and provide admissible bounds for the extension parameter \(k\). We first consider curves defined by \(x^{\ell_1}y^{\ell_2}-1\) and their monomial multiples; see Theorems~\ref{theo:case-2}, \ref{theo:case-3}, and \ref{theo:case-4}. We then treat the reducible curves defined by \(y(x-y^\ell)\) and, by symmetry, \(x(y-x^\ell)\); see Theorem~\ref{theo:case-5}. In each case, the proof reduces the truncated moment problem on the curve to a one-dimensional moment problem through a suitable parametrization, yielding explicit control of the required extension order. As a consequence, we obtain degree-bounded sums-of-squares certificates for polynomials that are strictly positive on the corresponding algebraic sets; see Corollaries~\ref{cor-after-4-1}, \ref{cor-after-4-4}, and \ref{cor-after-4-6}.\\

We begin with the irreducible curves defined by
\(x^{\ell_1}y^{\ell_2}=1\). Their Laurent parametrization reduces the
corresponding truncated moment problem to a strong truncated Hamburger
moment problem.

\begin{theorem}\label{theo:case-2}
Let $\ell_1,\ell_2\in\mathbb{N}\setminus\{1\}$ satisfy
$\gcd(\ell_1,\ell_2)=1$, and let
$
p(x,y):=x^{\ell_1}y^{\ell_2}-1.
$
Then, for every $n\geq \ell_1+\ell_2$, the generator family
$\{1,p,-p\}$ satisfies property $(S_{n,\ell_1+\ell_2-1})$.
\end{theorem}

\begin{proof}
Set \(k:=\ell_1+\ell_2-1\) and put \(N:=n+k\). Let
\(\beta\equiv\beta^{(2n)}\) be a truncated sequence having the
\(k\)-extension property relative to \(\{1,p,-p\}\). Thus, \(M(n)\)
admits a positive extension \(M(N)\) satisfying
\[
M_p(N)\succeq0,\qquad M_{-p}(N)\succeq0.
\]
Since \(M_{-p}(N)=-M_p(N)\), these inequalities imply
\(M_p(N)=0\). Equivalently, the columns of \(M(N)\) satisfy the
recursively generated relation \(X^{\ell_1}Y^{\ell_2}=1\) whenever
the corresponding products are defined. Let
\(L_\beta:\mathcal P_{2N}\to\mathbb R\) denote the Riesz functional
associated with the chosen extension.

We first record a consequence that will be used repeatedly:
\begin{equation}
\label{annihilation-p}
L_\beta(ph)=0
\qquad\text{whenever}\qquad
\deg(ph)\leq 2N.
\end{equation}
Indeed, it suffices by linearity to consider a monomial \(h\).
Under the indicated degree bound, we may factor \(h=qr\) so that
\(\deg q\leq N-\deg p\) and \(\deg r\leq N\). Recursive generation
gives \((pq)(X,Y)=0\), and pairing this column relation with
\(r(X,Y)\) gives \(L_\beta(ph)=L_\beta(pqr)=0\).

The curve \(\cZ_p\) is parametrized by
\begin{equation}
\label{parametrization-of-curve}
x=t^{\ell_2},\qquad y=t^{-\ell_1},
\qquad t\in\mathbb R\setminus\{0\}.
\end{equation}
Recall that a Laurent polynomial is a finite linear combination of
integer powers of \(t\), that is, an element of
\(\mathbb R[t,t^{-1}]\). If \(\nu\) is a positive Borel measure on
\(\mathbb R\setminus\{0\}\), its Laurent moments are
\[
\gamma_u:=\int_{\mathbb R\setminus\{0\}}t^u\,d\nu(t),
\qquad u\in\mathbb Z,
\]
whenever these integrals are finite. Positive exponents are the usual
moments, whereas negative exponents record moments of \(1/t\). More
generally, if the moments are specified only for \(u\) in a set
\(E\subseteq\mathbb Z\), we call
\(\{\gamma_u:u\in E\}\) a partial Laurent moment sequence.

Under the parametrization \eqref{parametrization-of-curve}, the
monomial \(x^iy^j\) becomes the Laurent monomial
\(t^{i\ell_2-j\ell_1}\). Consequently, the bivariate moment
\(\beta_{i,j}\) corresponds to the Laurent moment
\(\gamma_{i\ell_2-j\ell_1}\). The Laurent indices determined by
\(\beta^{(2n)}\) form the set
\[
D_{2n}
:=
\{i\ell_2-j\ell_1:
  i,j\in\mathbb N_0,\ i+j\leq2n\}
\subseteq\mathbb Z.
\]

For \(u\in D_{2n}\), choose \(i,j\in\mathbb N_0\) such that
\(u=i\ell_2-j\ell_1\) and \(i+j\leq2n\), and define
\(\gamma_u:=\beta_{i,j}\). This definition is independent of the
chosen representation. Indeed, suppose that
\(i\ell_2-j\ell_1=i'\ell_2-j'\ell_1\). Since
\(\gcd(\ell_1,\ell_2)=1\), there exists \(r\in\mathbb Z\) such that
\[
i-i'=r\ell_1,\qquad j-j'=r\ell_2.
\]
It follows that \(x^iy^j-x^{i'}y^{j'}\in(p)\). Moreover, both
monomials have degree at most \(2n\), so their difference can be
written as \(ph\) with \(\deg(ph)\leq2n\leq2N\). Hence
\eqref{annihilation-p} gives
\[
\beta_{i,j}-\beta_{i',j'}
=
L_\beta(x^iy^j-x^{i'}y^{j'})
=
0.
\]
Thus, \(\gamma=\{\gamma_u:u\in D_{2n}\}\) is a well-defined partial
Laurent sequence.

A representing measure for \(\gamma\) is a positive Borel measure
\(\nu\) on \(\mathbb R\setminus\{0\}\) satisfying
\[
\gamma_u
=
\int_{\mathbb R\setminus\{0\}}t^u\,d\nu(t),
\qquad u\in D_{2n}.
\]
Such a measure pushes forward under
\(\varphi(t):=(t^{\ell_2},t^{-\ell_1})\) to a
\(\cZ_p\)-representing measure for \(\beta^{(2n)}\).

The strong truncated Hamburger moment theorem that we shall use is
stated for Laurent moments indexed by a full interval of integers,
rather than by the possibly nonconsecutive set \(D_{2n}\). We
therefore construct an extension of \(\gamma\) to the interval
\[
I:=[-2n\ell_1-2,\,2n\ell_2+2]\cap\mathbb Z.
\]

We first claim that every integer
\(u\in[-n\ell_1-1,\,n\ell_2+1]\) can be written as
\[
u=i\ell_2-j\ell_1,
\qquad i,j\in\mathbb N_0,
\qquad i+j\leq N.
\]

Suppose first that \(1\leq u\leq n\ell_2\), and set
\(m:=\lceil u/\ell_2\rceil\). Then \(1\leq m\leq n\) and
\((m-1)\ell_2<u\leq m\ell_2\). Choose
\(r\in\{0,\ldots,\ell_1-1\}\) such that
\((m+r)\ell_2\equiv u\pmod{\ell_1}\), and set
\[
s:=\frac{(m+r)\ell_2-u}{\ell_1}.
\]
Then \(s\in\mathbb N_0\), and
\[
0\leq(m+r)\ell_2-u
<
(r+1)\ell_2
\leq\ell_1\ell_2.
\]
Consequently, \(s\leq\ell_2-1\), and
\(u=(m+r)\ell_2-s\ell_1\), where
\[
(m+r)+s
\leq
n+(\ell_1-1)+(\ell_2-1)
=
n+k-1
<
N+1.
\]

For the remaining positive endpoint \(u=n\ell_2+1\), choose
\(r\in\{1,\ldots,\ell_1-1\}\) such that
\(r\ell_2\equiv1\pmod{\ell_1}\), and set
\(s:=(r\ell_2-1)/\ell_1\). Then
\(0\leq s\leq\ell_2-1\), and
\[
n\ell_2+1
=
(n+r)\ell_2-s\ell_1,
\qquad
(n+r)+s\leq n+k=N.
\]

The negative indices are treated symmetrically. Suppose that
\(-n\ell_1\leq u\leq-1\), and put \(w:=-u\) and
\(m:=\lceil w/\ell_1\rceil\). Choose
\(r\in\{0,\ldots,\ell_2-1\}\) such that
\((m+r)\ell_1\equiv w\pmod{\ell_2}\), and set
\[
s:=\frac{(m+r)\ell_1-w}{\ell_2}.
\]
Then \(0\leq s\leq\ell_1-1\), and
\(u=s\ell_2-(m+r)\ell_1\), where
\[
s+(m+r)
\leq
n+(\ell_1-1)+(\ell_2-1)
=
n+k-1.
\]

Finally, for \(u=-n\ell_1-1\), choose
\(r\in\{1,\ldots,\ell_2-1\}\) such that
\(r\ell_1\equiv1\pmod{\ell_2}\), and set
\(s:=(r\ell_1-1)/\ell_2\). Then
\(0\leq s\leq\ell_1-1\), and
\[
-n\ell_1-1
=
s\ell_2-(n+r)\ell_1,
\qquad
s+(n+r)\leq n+k=N.
\]
The index \(u=0\) is represented by the monomial \(1\), proving the
claim.

For every
\[
u\in J:=[-n\ell_1-1,\,n\ell_2+1]\cap\mathbb Z,
\]
choose a monomial \(m_u:=X^{i_u}Y^{j_u}\) such that
\(\deg m_u\leq N\) and \(i_u\ell_2-j_u\ell_1=u\). Consider the
principal submatrix
\[
H
:=
\bigl(L_\beta(m_um_v)\bigr)_{u,v\in J}
\]
of \(M(N)\). Since \(M(N)\succeq0\), we have \(H\succeq0\).

We next show that the entries of \(H\) depend only on \(u+v\).
Suppose that \(u,v,u',v'\in J\) satisfy \(u+v=u'+v'\). Then
\(m_um_v\) and \(m_{u'}m_{v'}\) have the same Laurent index under
\eqref{parametrization-of-curve}. Their difference therefore belongs
to \((p)\). Since both products have degree at most \(2N\), we may
write \(m_um_v-m_{u'}m_{v'}=ph\) with \(\deg(ph)\leq2N\). By
\eqref{annihilation-p},
\[
L_\beta(m_um_v)
=
L_\beta(m_{u'}m_{v'}).
\]

Since
\[
J+J
=
[-2n\ell_1-2,\,2n\ell_2+2]\cap\mathbb Z
=
I,
\]
we may define, for \(w\in I\),
\[
\widetilde\gamma_w
:=
L_\beta(m_um_v),
\]
where \(u,v\in J\) are any indices satisfying \(u+v=w\). The
preceding argument shows that \(\widetilde\gamma_w\) is independent
of the chosen decomposition \(w=u+v\). Thus
\(\widetilde\gamma=\{\widetilde\gamma_w:w\in I\}\) is well defined,
and
\[
H
=
H_{\widetilde\gamma}
:=
\bigl(\widetilde\gamma_{u+v}\bigr)_{u,v\in J}
\succeq0.
\]

It remains to verify that \(\widetilde\gamma\) extends the original
partial sequence \(\gamma\). Let
\(q=i\ell_2-j\ell_1\in D_{2n}\), where \(i+j\leq2n\). Since
\[
D_{2n}
\subseteq
[-2n\ell_1,\,2n\ell_2]\cap\mathbb Z
\subseteq I,
\]
choose \(u,v\in J\) such that \(u+v=q\). The monomials \(X^iY^j\)
and \(m_um_v\) have the same Laurent index, so their difference
belongs to \((p)\). Moreover,
\[
\deg(X^iY^j)\leq2n\leq2N,
\qquad
\deg(m_um_v)\leq2N.
\]
Hence \eqref{annihilation-p} gives
\[
\widetilde\gamma_q
=
L_\beta(m_um_v)
=
L_\beta(X^iY^j)
=
\beta_{i,j}
=
\gamma_q.
\]
Therefore, \(\widetilde\gamma\) is a Laurent extension of \(\gamma\).

Let
\[
\widehat\gamma
:=
\{\widetilde\gamma_u:
  -2n\ell_1\leq u\leq2n\ell_2\},
\]
and let
\[
H_{\widehat\gamma}
:=
\bigl(\widetilde\gamma_{u+v}\bigr)_{
-n\ell_1\leq u,v\leq n\ell_2}
\]
be its Laurent Hankel matrix. Since \(H_{\widehat\gamma}\) is a
principal submatrix of \(H_{\widetilde\gamma}\), we have
\(H_{\widehat\gamma}\succeq0\).
Suppose first that
\(\widetilde\gamma_{-2n\ell_1}=0\).
This is the diagonal entry of \(H_{\widetilde\gamma}\) indexed by
\(-n\ell_1\). Since \(H_{\widetilde\gamma}\succeq0\), the corresponding
row and column vanish. It is easy to show inductively that 
$H_{\widehat \gamma}=0$, and hence \(\widehat\gamma\) is the zero
sequence. Since \(\widehat\gamma\) extends the original partial sequence
\(\gamma\), it follows that \(\beta^{(2n)}\) is the zero sequence, which
is represented by the zero measure. We may therefore assume from now on
that
$
\widetilde\gamma_{-2n\ell_1}>0.
$

If \(H_{\widehat\gamma}\succ0\), then
\cite[Theorem~3.1(4a)]{Zal22j} implies that
\(\widehat\gamma\) admits a representing measure supported on
\(\mathbb R\setminus\{0\}\). Suppose now that
\(H_{\widehat\gamma}\) is singular. The principal submatrices of
\(H_{\widetilde\gamma}\) indexed by
\[
[-n\ell_1,\,n\ell_2+1]
\qquad\text{and}\qquad
[-n\ell_1-1,\,n\ell_2]
\]
are positive semidefinite one-step extensions of
\(H_{\widehat\gamma}\) on the positive and negative sides,
respectively. By \cite[Proposition~2.1(4),(5)]{Zal22j},
\[
\operatorname{rank}H_{\widehat\gamma}
=
\operatorname{rank}
\bigl(\widetilde\gamma_{u+v}\bigr)_{
-n\ell_1\leq u,v\leq n\ell_2-1}
=
\operatorname{rank}
\bigl(\widetilde\gamma_{u+v}\bigr)_{
-n\ell_1+1\leq u,v\leq n\ell_2}.
\]
Hence condition \cite[Theorem~3.1(4b)]{Zal22j} is satisfied.
Therefore, in either case, \cite[Theorem~3.1]{Zal22j} yields a
representing measure \(\nu\) for \(\widehat\gamma\), supported on
\(\mathbb R\setminus\{0\}\).

Push \(\nu\) forward under
\(t\mapsto(t^{\ell_2},t^{-\ell_1})\). The resulting measure \(\mu\)
is supported on \(\cZ_p\), and for \(i,j\in\mathbb N_0\) with
\(i+j\leq2n\),
\[
\int_{\cZ_p}x^iy^j\,d\mu
=
\int_{\mathbb R\setminus\{0\}}
t^{i\ell_2-j\ell_1}\,d\nu(t)
=
\widehat\gamma_{i\ell_2-j\ell_1}
=
\gamma_{i\ell_2-j\ell_1}
=
\beta_{i,j}.
\]
Thus, \(\mu\) is a \(\cZ_p\)-representing measure for
\(\beta^{(2n)}\).

We have proved that the \(k\)-extension property relative to
\(\{1,p,-p\}\), where \(k=\ell_1+\ell_2-1\), implies the existence
of a \(\cZ_p\)-representing measure. The converse follows from the
standard necessity of the positive extension and localizing
conditions. Therefore, \(\{1,p,-p\}\) satisfies property
$
(S_{n,\ell_1+\ell_2-1}).
$
\end{proof}

\begin{corollary}
\label{cor-after-4-1}
Let \(\ell_1,\ell_2\in\NN\setminus\{1\}\) satisfy
\(\gcd(\ell_1,\ell_2)=1\), and let
\[
        p(x,y):=x^{\ell_1}y^{\ell_2}-1.
\]
Suppose that \(n\ge\ell_1+\ell_2\). Then every polynomial
\(f\in\mathcal P_{2n}\) that is strictly positive on \(\cZ_p\)
admits a representation
\[
        f
        =
        \sum_{j=1}^r g_j^2+ph,
\]
where \(g_1,\ldots,g_r,h\in\mathbb R[x,y]\) and
$
        g_j^2,\ ph
        \in
        \mathcal P_{
        2\left(
        n+\ell_1+\ell_2-1
        \right)},$
        $j=1,\ldots,r.
$
\end{corollary}

\begin{proof}
Set
$
k
        :=\ell_1+\ell_2-1.
$
By Theorem~\ref{theo:case-2}, the generator family
\(\{1,p,-p\}\) satisfies property \((S_{n,k})\). Hence,
Theorem~\ref{thm:Psatz-bounds}\textup{(i)} implies that
\(f\in\Sigma_{\{1,p,-p\},n+k}\). Thus
$
        f
        =
        \sigma_0+p\sigma_1-p\sigma_2
$
for some sums of squares \(\sigma_0,\sigma_1,\sigma_2\) satisfying
the corresponding degree bounds. Writing
\(\sigma_0=\sum_{j=1}^r g_j^2\) and
\(h:=\sigma_1-\sigma_2\), we obtain
$
        f=\sum_{j=1}^r g_j^2+ph,
$
with \(g_j^2,ph\in\mathcal P_{2(n+k)}\), as required.
\end{proof}

We next extend Theorem~\ref{theo:case-2} to reducible curves obtained
by adjoining one of the coordinate axes to
\(\cZ_{x^{\ell_1}y^{\ell_2}-1}\).

\begin{theorem}
\label{theo:case-3}
Let \(\ell_1,\ell_2\in\NN\setminus\{1\}\) satisfy
\(\gcd(\ell_1,\ell_2)=1\), and let
\[
        p(x,y)
        :=
        q(x,y)\bigl(x^{\ell_1}y^{\ell_2}-1\bigr),
        \qquad
        q(x,y)\in\{x,y\}.
\]
Then, for every \(n\ge\ell_1+\ell_2+1\), the generator family
\(\{1,p,-p\}\) satisfies property
$
        \left(
        S_{n,\; 2(\ell_1+\ell_2)}
        \right).
$
\end{theorem}

\begin{proof}
By symmetry, it suffices to consider
\(p(x,y)=y(x^{\ell_1}y^{\ell_2}-1)\). Set
\(s:=\ell_1+\ell_2\), \(k:=2s\), \(N:=n+k=n+2s\), and
\(g(x,y):=x^{\ell_1}y^{\ell_2}-1\).

Let \(\beta\equiv\beta^{(2n)}\) have the \(k\)-extension property
relative to \(\{1,p,-p\}\). Thus, \(M(n)\) admits a positive
extension \(M(N)\) satisfying \(M_p(N)\succeq0\) and
\(M_{-p}(N)\succeq0\). Since \(M_{-p}(N)=-M_p(N)\), it follows that
\(M_p(N)=0\). By polarization and recursive generation, the columns
of \(M(N)\) satisfy
\(X^iY^{j+1}=X^{i+\ell_1}Y^{j+\ell_2+1}\) whenever the monomials
involved have degree at most \(N\).

We shall use the following consequence of this column relation:
\(L_\beta(pr)=0\) whenever \(\deg(pr)\leq2N\). Indeed, it is enough
to consider a monomial \(r\). Under the indicated degree bound, we
may factor \(r=r_1r_2\) so that \(\deg(pr_1)\leq N\) and
\(\deg r_2\leq N\). The recursively generated relation gives
\((pr_1)(X)=0\), and pairing this column relation with \(r_2(X)\)
yields \(L_\beta(pr)=L_\beta(pr_1r_2)=0\). Linearity gives the
assertion for arbitrary polynomials \(r\) in the same degree range.

The zero set of \(p\) is
\(\mathcal Z_p=\mathcal Z_y\cup\mathcal Z_g\). We shall decompose
the degree-\(2n\) truncation of \(M(N)\) into positive moment
matrices corresponding to these two components.

Order the monomials of degree at most \(N\) as
\(\vec X^{(0,N)},\vec{\mathcal T}_N\), where
\(\vec X^{(0,N)}=(1,X,\ldots,X^N)\) and
\(\vec{\mathcal T}_N\) consists of the monomials of degree at most
\(N\) containing a positive power of \(Y\). With respect to this
ordering, write
\[
M(N)=
\begin{pmatrix}
A^{(N)} & B^{(N)}\\
(B^{(N)})^{\mathsf T} & C^{(N)}
\end{pmatrix}.
\]
Thus, the rows and columns of \(A^{(N)}\) are indexed by the
monomials \(X^a\), \(0\leq a\leq N\), where \(X^0=1\); the rows of
\(B^{(N)}\) have the same indexing, while its columns and the rows
and columns of \(C^{(N)}\) are indexed by the monomials in
\(\mathcal T_N\).

Since \(M(N)\succeq0\), the generalized Schur-complement criterion
gives
\(\operatorname{Ran}((B^{(N)})^{\mathsf T})
 \subseteq\operatorname{Ran}C^{(N)}\) and
\(A^{(N)}-B^{(N)}(C^{(N)})^\dagger(B^{(N)})^{\mathsf T}\succeq0\).
Define
\(A_1^{(N)}:=B^{(N)}(C^{(N)})^\dagger(B^{(N)})^{\mathsf T}\).
We use monomial subscripts for its entries; thus,
\((A_1^{(N)})_{X^a,X^b}\) denotes the entry in the row indexed by
\(X^a\) and the column indexed by \(X^b\). In particular,
\((A_1^{(N)})_{X^a,1}\) means
\((A_1^{(N)})_{X^a,X^0}\). We have
\begin{equation}
\label{eq:case-3-schur-decomposition}
M(N)=
\begin{pmatrix}
A^{(N)}-A_1^{(N)} & 0\\
0 & 0
\end{pmatrix}
+
\begin{pmatrix}
A_1^{(N)} & B^{(N)}\\
(B^{(N)})^{\mathsf T} & C^{(N)}
\end{pmatrix},
\end{equation}
and both summands are positive semidefinite. We shall show that,
after suitable truncation, the first summand corresponds to
\(\mathcal Z_y\), while the second corresponds to \(\mathcal Z_g\).

Set \(R:=N-s=n+s\). For each \(a\in\{0,\ldots,R\}\), define
\(z_a:=X^{a+\ell_1}Y^{\ell_2}\). Since
\(\deg z_a=a+s\leq R+s=N\), the monomial \(z_a\) belongs to
\(\mathcal T_N\).

Let \(r\in\mathcal T_N\). Since \(r\) contains a positive power of
\(Y\), the quotient \(r/Y\) is a monomial, and
\(z_ar-X^ar=X^ar(X^{\ell_1}Y^{\ell_2}-1)=pX^a(r/Y)\). Moreover,
\(\deg(z_ar)\leq R+s+N=2N\). It follows that
\(L_\beta(X^ar)=L_\beta(z_ar)\). In terms of the blocks
\(B^{(N)}\) and \(C^{(N)}\), this means
\begin{equation}
\label{eq:case-3-block-row-identity}
B^{(N)}_{X^a,\bullet}=C^{(N)}_{z_a,\bullet},
\qquad 0\leq a\leq R,
\end{equation}
where the right-hand side denotes the row of \(C^{(N)}\) indexed by
\(z_a\).

Using \eqref{eq:case-3-block-row-identity} and
\(C^{(N)}(C^{(N)})^\dagger C^{(N)}=C^{(N)}\), we obtain, for
\(0\leq a,b\leq R\),
\begin{equation}
\label{eq:case-3-schur-block-entries}
\begin{aligned}
(A_1^{(N)})_{X^a,X^b}
&=B^{(N)}_{X^a,\bullet}(C^{(N)})^\dagger
  (B^{(N)}_{X^b,\bullet})^{\mathsf T}\\
&=e_{z_a}^{\mathsf T}
  C^{(N)}(C^{(N)})^\dagger C^{(N)}e_{z_b}\\
&=C^{(N)}_{z_a,z_b}.
\end{aligned}
\end{equation}
In particular,
\((A_1^{(N)})_{X^a,X^b}
 =\beta_{a+b+2\ell_1,\,2\ell_2}\), so these entries depend only on
\(a+b\). Thus the principal submatrix of \(A_1^{(N)}\) indexed by
\(1,X,\ldots,X^R\) is Hankel.

Let
\[
M^{[g]}(R):=
\begin{pmatrix}
A_1^{(R)} & B^{(R)}\\
(B^{(R)})^{\mathsf T} & C^{(R)}
\end{pmatrix}
\]
be the principal submatrix of the second summand in
\eqref{eq:case-3-schur-decomposition} indexed by all monomials of
degree at most \(R\), where \(A_1^{(R)},B^{(R)}\), and \(C^{(R)}\)
denote the corresponding restrictions of the blocks in
\eqref{eq:case-3-schur-decomposition}. Since the second summand in
\eqref{eq:case-3-schur-decomposition} is positive semidefinite,
\(M^{[g]}(R)\succeq0\). The entries of \(B^{(R)}\) and \(C^{(R)}\) are
entries of the original moment matrix \(M(N)\), while
\eqref{eq:case-3-schur-block-entries} shows that \(A_1^{(R)}\) is
Hankel. Consequently, \(M^{[g]}(R)\) is a well-defined bivariate moment
matrix.

We next show that its columns satisfy
\(X^{\ell_1}Y^{\ell_2}=1\). Write
\(z_0=X^{\ell_1}Y^{\ell_2}\). If the row is indexed by a pure
monomial \(X^a\), where \(0\leq a\leq R\), then
\eqref{eq:case-3-block-row-identity} and
\eqref{eq:case-3-schur-block-entries} give
\[
(M^{[g]}(R))_{X^a,1}
=(A_1^{(N)})_{X^a,1}
=C^{(N)}_{z_a,z_0}
=B^{(N)}_{X^a,z_0}
=(M^{[g]}(R))_{X^a,z_0}.
\]
If the row is indexed by a monomial \(r\in\mathcal T_R\), then
\eqref{eq:case-3-block-row-identity}, with \(a=0\), gives
\((M^{[g]}(R))_{r,1}=B^{(N)}_{1,r}=C^{(N)}_{z_0,r}
=(M^{[g]}(R))_{r,z_0}\). Thus the columns indexed by \(1\) and
\(X^{\ell_1}Y^{\ell_2}\) coincide. Since \(M^{[g]}(R)\) is positive
semidefinite, the standard kernel property for moment matrices shows
that this relation is recursively generated.

The degree-\(2n\) truncation of \(M^{[g]}(R)\) therefore has a positive
extension to order \(n+s-1\leq R\) satisfying
\(X^{\ell_1}Y^{\ell_2}=1\). Equivalently, the localizing matrices
associated with \(g\) and \(-g\) vanish. Hence this truncation has
the \((s-1)\)-extension property relative to \(\{1,g,-g\}\). By
Theorem~\ref{theo:case-2}, it admits a representing measure
\(\mu_g\) supported on
\(\mathcal Z_g=\mathcal Z_{x^{\ell_1}y^{\ell_2}-1}\).

It remains to treat the first summand in
\eqref{eq:case-3-schur-decomposition}. Set
\(D^{(N)}:=A^{(N)}-A_1^{(N)}\). Then \(D^{(N)}\succeq0\).
Moreover, \(A^{(N)}\) is Hankel, and
\eqref{eq:case-3-schur-block-entries} shows that the entries of
\(A_1^{(N)}\) depend only on the sum of their exponents whenever the
corresponding monomials have degree at most \(R\). It follows that
\[
\bigl(D^{(N)}_{X^a,X^b}\bigr)_{0\leq a,b\leq R}
\]
is a positive semidefinite univariate Hankel moment matrix.

Since \(n+1\leq R\), the principal submatrix
\[
D^{(n+1)}
:=
\bigl(D^{(N)}_{X^a,X^b}\bigr)_{0\leq a,b\leq n+1}
\]
is a positive semidefinite univariate Hankel moment matrix extending
\(D^{(n)}
 :=(D^{(N)}_{X^a,X^b})_{0\leq a,b\leq n}\). Since the generator
family \(\{1\}\) satisfies property \((S_{n,1})\), the degree-\(2n\)
truncation corresponding to \(D^{(n)}\) admits a representing
measure \(\nu\) on \(\mathbb R\). Pushing \(\nu\) forward under
\(t\mapsto(t,0)\) gives a measure \(\mu_y\) supported on
\(\mathcal Z_y\).

Restricting \eqref{eq:case-3-schur-decomposition} to moments of
degree at most \(2n\) gives
\(\beta^{(2n)}=\beta_y^{(2n)}+\beta_g^{(2n)}\), where
\(\beta_y^{(2n)}\) and \(\beta_g^{(2n)}\) are represented by
\(\mu_y\) and \(\mu_g\), respectively. Therefore
\(\mu:=\mu_y+\mu_g\) represents \(\beta^{(2n)}\), and
\(\operatorname{supp}\mu\subseteq
\mathcal Z_y\cup\mathcal Z_g=\mathcal Z_p\).

We have proved that the \(2(\ell_1+\ell_2)\)-extension property
relative to \(\{1,p,-p\}\) implies the existence of a
\(\mathcal Z_p\)-representing measure. The converse follows from the
standard necessity of positive extensions and the corresponding
localizing conditions. Hence \(\{1,p,-p\}\) satisfies property
\(\left(S_{n,\,2(\ell_1+\ell_2)}\right)\).
\end{proof}

We finally consider the case in which both coordinate axes are
adjoined to the curve
\(\cZ_{x^{\ell_1}y^{\ell_2}-1}\).

\begin{theorem}
\label{theo:case-4}
Let \(\ell_1,\ell_2\in\NN\setminus\{1\}\) satisfy
\(\gcd(\ell_1,\ell_2)=1\), and let
\[
        p(x,y)
        :=
        xy\bigl(x^{\ell_1}y^{\ell_2}-1\bigr).
\]
Then, for every \(n\ge\ell_1+\ell_2+2\), the generator family
\(\{1,p,-p\}\) satisfies property
$
        \left(
        S_{n,\,2(\ell_1+\ell_2)}
        \right).
$
\end{theorem}

\begin{proof}
Set \(s:=\ell_1+\ell_2\), \(k:=2s\), \(N:=n+k=n+2s\), and
\(g(x,y):=x^{\ell_1}y^{\ell_2}-1\).

Let \(\beta\equiv\beta^{(2n)}\) have the \(k\)-extension property
relative to \(\{1,p,-p\}\). Thus, \(M(n)\) admits a positive
extension \(M(N)\) satisfying \(M_p(N)\succeq0\) and
\(M_{-p}(N)\succeq0\). Since \(M_{-p}(N)=-M_p(N)\), we have
\(M_p(N)=0\). By polarization and recursive generation, the columns
of \(M(N)\) satisfy
\(X^{i+1}Y^{j+1}=X^{i+\ell_1+1}Y^{j+\ell_2+1}\) whenever the
monomials involved have degree at most \(N\).

As in the proof of Theorem~\ref{theo:case-3}, the corresponding
factorization argument gives
\[
L_\beta(pr)=0
\qquad\text{whenever}\qquad
\deg(pr)\leq2N.
\tag{\(\ast\)}
\]
Indeed, one factors each monomial \(r=r_1r_2\) so that
\(\deg(pr_1)\leq N\) and \(\deg r_2\leq N\), and then pairs the
column relation \((pr_1)(X)=0\) with \(r_2(X)\).

The zero set of \(p\) is
\(\cZ_p=\cZ_x\cup\cZ_y\cup\cZ_g\). We shall decompose the
degree-\(2n\) truncation of \(M(N)\) into positive moment matrices
corresponding to these three components.

Order the monomials of degree at most \(N\) as
\(1,\vec X^{(1,N)},\vec Y^{(1,N)},\vec{\mathcal T}_N\), where
\(\vec X^{(1,N)}=(X,\ldots,X^N)\),
\(\vec Y^{(1,N)}=(Y,\ldots,Y^N)\), and
\(\vec{\mathcal T}_N\) consists of the mixed monomials
\(X^iY^j\) with \(i,j\geq1\) and \(i+j\leq N\). With respect to
this ordering, write
\[
M(N)=
\begin{pmatrix}
\beta_{0,0}
&
(a^{(N)})^{\mathsf T}
&
(b^{(N)})^{\mathsf T}
&
(c^{(N)})^{\mathsf T}
\\
a^{(N)}
&
A^{(N)}
&
B^{(N)}
&
C^{(N)}
\\
b^{(N)}
&
(B^{(N)})^{\mathsf T}
&
D^{(N)}
&
E^{(N)}
\\
c^{(N)}
&
(C^{(N)})^{\mathsf T}
&
(E^{(N)})^{\mathsf T}
&
F^{(N)}
\end{pmatrix}.
\]
Thus, \(A^{(N)}\), \(D^{(N)}\), and \(F^{(N)}\) are indexed by the
monomials \(X^a\), \(Y^b\), and the mixed monomials, respectively.
We use monomial subscripts throughout; for example,
\(A^{(N)}_{X^a,X^{a'}}\) denotes the entry in the row indexed by
\(X^a\) and the column indexed by \(X^{a'}\).

Set
\[
Q^{(N)}:=
\begin{pmatrix}
\beta_{0,0} & (a^{(N)})^{\mathsf T} & (b^{(N)})^{\mathsf T}\\
a^{(N)} & A^{(N)} & B^{(N)}\\
b^{(N)} & (B^{(N)})^{\mathsf T} & D^{(N)}
\end{pmatrix},
\qquad
G^{(N)}:=
\begin{pmatrix}
(c^{(N)})^{\mathsf T}\\
C^{(N)}\\
E^{(N)}
\end{pmatrix}.
\]
Then
\(M(N)=\begin{pmatrix}Q^{(N)}&G^{(N)}\\
(G^{(N)})^{\mathsf T}&F^{(N)}\end{pmatrix}\). Since
\(M(N)\succeq0\), the generalized Schur-complement criterion gives
\[
S^{(N)}
:=
Q^{(N)}
-
G^{(N)}(F^{(N)})^\dagger(G^{(N)})^{\mathsf T}
\succeq0.
\]
Consequently,
\begin{equation}
\label{eq:case-4-schur-decomposition}
M(N)=
\begin{pmatrix}
S^{(N)} & 0\\
0 & 0
\end{pmatrix}
+
\begin{pmatrix}
Q_1^{(N)} & G^{(N)}\\
(G^{(N)})^{\mathsf T} & F^{(N)}
\end{pmatrix},
\end{equation}
where
\(Q_1^{(N)}
 :=G^{(N)}(F^{(N)})^\dagger(G^{(N)})^{\mathsf T}\), and both
summands are positive semidefinite.

Write
\[
Q_1^{(N)}=
\begin{pmatrix}
\eta_3
&
(a_1^{(N)})^{\mathsf T}
&
(b_1^{(N)})^{\mathsf T}
\\
a_1^{(N)}
&
A_1^{(N)}
&
B_1^{(N)}
\\
b_1^{(N)}
&
(B_1^{(N)})^{\mathsf T}
&
D_1^{(N)}
\end{pmatrix}.
\]
Thus,
\[
\begin{aligned}
\eta_3
&:=
(c^{(N)})^{\mathsf T}(F^{(N)})^\dagger c^{(N)},\\
a_1^{(N)}
&:=
C^{(N)}(F^{(N)})^\dagger c^{(N)},
&\qquad
b_1^{(N)}
&:=
E^{(N)}(F^{(N)})^\dagger c^{(N)},\\
A_1^{(N)}
&:=
C^{(N)}(F^{(N)})^\dagger(C^{(N)})^{\mathsf T},
&
B_1^{(N)}
&:=
C^{(N)}(F^{(N)})^\dagger(E^{(N)})^{\mathsf T},\\
D_1^{(N)}
&:=
E^{(N)}(F^{(N)})^\dagger(E^{(N)})^{\mathsf T}.
\end{aligned}
\]

Set \(R:=N-s=n+s\). For \(0\leq a,b\leq R\), define
\(z_a:=X^{a+\ell_1}Y^{\ell_2}\) and
\(w_b:=X^{\ell_1}Y^{b+\ell_2}\). Then \(z_0=w_0\), and all these
monomials have degree at most \(N\).

The argument used to obtain the block-row identity in the proof of
Theorem~\ref{theo:case-3} applies in both coordinate directions.
Indeed, if \(r\in\mathcal T_N\), then
\[
z_ar-X^ar=pX^a\frac{r}{XY},
\qquad
w_br-Y^br=pY^b\frac{r}{XY},
\]
and the degrees are at most \(2N\). Hence \((\ast)\) gives
\begin{equation}
\label{eq:case-4-mixed-row-identities}
(c^{(N)})^{\mathsf T}=F^{(N)}_{z_0,\bullet},
\qquad
C^{(N)}_{X^a,\bullet}=F^{(N)}_{z_a,\bullet},
\qquad
E^{(N)}_{Y^b,\bullet}=F^{(N)}_{w_b,\bullet},
\end{equation}
for \(1\leq a,b\leq R\).

Applying the same Schur-completion calculation as in the proof of
Theorem~\ref{theo:case-3}, now using
\(F^{(N)}(F^{(N)})^\dagger F^{(N)}=F^{(N)}\), gives
\begin{equation}
\label{eq:case-4-completed-block-entries}
\begin{aligned}
\eta_3
&=F^{(N)}_{z_0,z_0},\\
(a_1^{(N)})_{X^a}
&=F^{(N)}_{z_a,z_0},
&
(b_1^{(N)})_{Y^b}
&=F^{(N)}_{w_b,z_0},\\
(A_1^{(N)})_{X^a,X^{a'}}
&=F^{(N)}_{z_a,z_{a'}},
&
(D_1^{(N)})_{Y^b,Y^{b'}}
&=F^{(N)}_{w_b,w_{b'}},\\
(B_1^{(N)})_{X^a,Y^b}
&=F^{(N)}_{z_a,w_b},
\end{aligned}
\end{equation}
whenever the exponents lie in \(\{1,\ldots,R\}\).

It remains to identify the completed \(X\)-\(Y\) block. For
\(1\leq a,b\leq R\),
\[
z_aw_b-X^aY^b
=
pX^{a-1}Y^{b-1}
\bigl(X^{\ell_1}Y^{\ell_2}+1\bigr),
\]
and \(\deg(z_aw_b)=a+b+2s\leq2N\). Thus \((\ast)\) and
\eqref{eq:case-4-completed-block-entries} give
\begin{equation}
\label{eq:case-4-cross-block-identity}
(B_1^{(N)})_{X^a,Y^b}
=
F^{(N)}_{z_a,w_b}
=
B^{(N)}_{X^a,Y^b},
\qquad 1\leq a,b\leq R.
\end{equation}

Let \(M^{[g]}(R)\) be the principal submatrix of the second summand in
\eqref{eq:case-4-schur-decomposition} indexed by all monomials of
degree at most \(R\). Then \(M^{[g]}(R)\succeq0\). The identities in
\eqref{eq:case-4-completed-block-entries} show that its completed
pure-\(X\) and pure-\(Y\) blocks have the required Hankel structure,
while \eqref{eq:case-4-cross-block-identity} identifies its
\(X\)-\(Y\) block with the corresponding block of \(M(N)\).
Consequently, \(M^{[g]}(R)\) is a well-defined bivariate moment matrix.

The verification that its columns satisfy
\(X^{\ell_1}Y^{\ell_2}=1\) is the same as in the proof of
Theorem~\ref{theo:case-3}. Namely,
\eqref{eq:case-4-mixed-row-identities} and
\eqref{eq:case-4-completed-block-entries} show that the columns
indexed by \(1\) and \(z_0=X^{\ell_1}Y^{\ell_2}\) agree on the
constant, pure-\(X\), pure-\(Y\), and mixed rows. Positivity then
implies that this relation is recursively generated.

The degree-\(2n\) truncation of \(M^{[g]}(R)\) therefore has a positive
extension to order \(n+s-1\leq R\) satisfying
\(X^{\ell_1}Y^{\ell_2}=1\). Hence it has the
\((s-1)\)-extension property relative to \(\{1,g,-g\}\). By
Theorem~\ref{theo:case-2}, it admits a representing measure
\(\mu_g\) supported on
\(\cZ_g=\cZ_{x^{\ell_1}y^{\ell_2}-1}\).

It remains to split the first summand in
\eqref{eq:case-4-schur-decomposition} between the coordinate axes.
Let \(S^{(R)}\) be the principal submatrix of \(S^{(N)}\) indexed
by \(1,X,\ldots,X^R,Y,\ldots,Y^R\). 
Set
\[
\begin{aligned}
\rho
&:=\beta_{0,0}-\eta_3,\\
u^{(R)}
&:=a^{(R)}-a_1^{(R)},
&\qquad
v^{(R)}
&:=b^{(R)}-b_1^{(R)},\\
P_X^{(R)}
&:=A^{(R)}-A_1^{(R)},
&
P_Y^{(R)}
&:=D^{(R)}-D_1^{(R)}.
\end{aligned}
\]
By
\eqref{eq:case-4-cross-block-identity},
\begin{equation}
\label{eq:case-4-residual-block}
S^{(R)}
=
\begin{pmatrix}
\rho & (u^{(R)})^{\mathsf T} & (v^{(R)})^{\mathsf T}\\
u^{(R)} & P_X^{(R)} & 0\\
v^{(R)} & 0 & P_Y^{(R)}
\end{pmatrix}
\succeq0.
\end{equation}

The generalized Schur-complement criterion applied to
\eqref{eq:case-4-residual-block} gives
\(\rho\geq\alpha_X+\alpha_Y\), where
\[
\alpha_X
:=
(u^{(R)})^{\mathsf T}(P_X^{(R)})^\dagger u^{(R)},
\qquad
\alpha_Y
:=
(v^{(R)})^{\mathsf T}(P_Y^{(R)})^\dagger v^{(R)}.
\]
Choose \(\eta_1,\eta_2\geq0\) such that
\(\eta_1\geq\alpha_Y\), \(\eta_2\geq\alpha_X\), and
\(\eta_1+\eta_2=\rho\). Then
\[
H_x(R):=
\begin{pmatrix}
\eta_1 & (v^{(R)})^{\mathsf T}\\
v^{(R)} & P_Y^{(R)}
\end{pmatrix}
\succeq0,
\qquad
H_y(R):=
\begin{pmatrix}
\eta_2 & (u^{(R)})^{\mathsf T}\\
u^{(R)} & P_X^{(R)}
\end{pmatrix}
\succeq0.
\]
Here \(H_x(R)\) is indexed by \(1,Y,\ldots,Y^R\), while
\(H_y(R)\) is indexed by \(1,X,\ldots,X^R\).

As in the treatment of the axis component in the proof of
Theorem~\ref{theo:case-3},
\eqref{eq:case-4-completed-block-entries} shows that these are
univariate Hankel moment matrices. Since \(n+1\leq R\), their
principal submatrices of order \(n+1\) give positive one-step
extensions of their order-\(n\) truncations. Property \((S_{n,1})\)
therefore yields representing measures on \(\mathbb R\). Pushing
the measure corresponding to \(H_x(R)\) forward under
\(t\mapsto(0,t)\) gives a measure \(\mu_x\) supported on \(\cZ_x\),
while pushing the measure corresponding to \(H_y(R)\) forward under
\(t\mapsto(t,0)\) gives a measure \(\mu_y\) supported on \(\cZ_y\).

Restricting \eqref{eq:case-4-schur-decomposition} and
\eqref{eq:case-4-residual-block} to moments of degree at most \(2n\)
gives
\(\beta^{(2n)}
 =\beta_x^{(2n)}+\beta_y^{(2n)}+\beta_g^{(2n)}\).
Consequently, \(\mu:=\mu_x+\mu_y+\mu_g\) represents
\(\beta^{(2n)}\), and
\(\operatorname{supp}\mu
 \subseteq\cZ_x\cup\cZ_y\cup\cZ_g=\cZ_p\).

We have shown that the \(2(\ell_1+\ell_2)\)-extension property
relative to \(\{1,p,-p\}\) implies the existence of a
\(\cZ_p\)-representing measure. The converse follows from the
standard necessity of positive extensions and the corresponding
localizing conditions. Hence \(\{1,p,-p\}\) satisfies property
\(\left(S_{n,\,2(\ell_1+\ell_2)}\right)\).
\end{proof}

\begin{corollary}
\label{cor-after-4-4}
Let \(\ell_1,\ell_2\in\NN\setminus\{1\}\) satisfy
\(\gcd(\ell_1,\ell_2)=1\), and let
\[
        p(x,y)
        :=
        q(x,y)\bigl(x^{\ell_1}y^{\ell_2}-1\bigr),
        \qquad
        q(x,y)\in\{x,y,xy\}.
\]
Suppose that \(n\ge\deg p\). Then every
\(f\in\mathcal P_{2n}\) that is strictly positive on \(\cZ_p\)
admits a representation
\[
        f=\sum_{j=1}^r g_j^2+ph,
\]
where \(g_1,\ldots,g_r,h\in\mathbb R[x,y]\) and
$
        g_j^2,\ ph
        \in
        \mathcal P_{
        2\left(
        n+2(\ell_1+\ell_2)\right)}$,
        $
        j=1,\ldots,r.
$
\end{corollary}

The proof is identical to that of
Corollary~\ref{cor-after-4-1}, using
Theorem~\ref{theo:case-3} when \(q\in\{x,y\}\) and
Theorem~\ref{theo:case-4} when \(q=xy\).\\

We conclude the section with the reducible curve obtained by adjoining
the \(x\)-axis to \(\cZ_{x-y^\ell}\).

\begin{theorem}
\label{theo:case-5}
Let \(p(x,y):=y(x-y^\ell)\), where \(\ell\in\NN\). Then, for every
\(n\ge\ell+1\), the generator family \(\{1,p,-p\}\) satisfies
property
$
        \left(
        S_{n,\ell+\max\{\ell-1,2\}}
        \right).
$
\end{theorem}

\begin{proof}
Set \(h:=\max\{\ell-1,2\}\), \(k:=\ell+h\), and
\(N:=n+k=n+\ell+h\). Also, put \(g(x,y):=x-y^\ell\).

Let \(\beta\equiv\beta^{(2n)}\) have the \(k\)-extension property
relative to \(\{1,p,-p\}\). Thus, \(M(n)\) admits a positive
extension \(M(N)\) satisfying \(M_p(N)\succeq0\) and
\(M_{-p}(N)\succeq0\). Since \(M_{-p}(N)=-M_p(N)\), it follows that
\(M_p(N)=0\). By polarization and recursive generation, the columns
of \(M(N)\) satisfy
\[
X^{i+1}Y^{j+1}=X^iY^{j+\ell+1}
\]
whenever the monomials involved have degree at most \(N\).

As in the proof of Theorem~\ref{theo:case-3}, the corresponding
factorization argument gives
\begin{equation}
\label{eq:case-5-saturation}
L_\beta(pr)=0
\qquad\text{whenever}\qquad
\deg(pr)\leq2N.
\end{equation}
Indeed, one factors each monomial \(r=r_1r_2\) so that
\(\deg(pr_1)\leq N\) and \(\deg r_2\leq N\), and then pairs the
column relation \((pr_1)(X)=0\) with \(r_2(X)\).

The zero set of \(p\) is
\(\cZ_p=\cZ_y\cup\cZ_g\). We shall decompose the degree-\(2n\)
truncation of \(M(N)\) into positive moment matrices corresponding
to these two components.

Order the monomials of degree at most \(N\) as
\(\vec X^{(0,N)},\vec{\mathcal T}_N\), where
\(\vec X^{(0,N)}=(1,X,\ldots,X^N)\) and
\(\vec{\mathcal T}_N\) consists of the monomials of degree at most
\(N\) containing a positive power of \(Y\). With respect to this
ordering, write
\[
M(N)=
\begin{pmatrix}
A^{(N)} & B^{(N)}\\
(B^{(N)})^{\mathsf T} & C^{(N)}
\end{pmatrix}.
\]
Thus, the rows and columns of \(A^{(N)}\) are indexed by the
monomials \(X^a\), \(0\leq a\leq N\), where \(X^0=1\), while the
rows and columns of \(C^{(N)}\) are indexed by the monomials in
\(\mathcal T_N\). We use monomial subscripts throughout; for
example, \((A^{(N)})_{X^a,X^b}\) denotes the entry in the row
indexed by \(X^a\) and the column indexed by \(X^b\).

Since \(M(N)\succeq0\), the generalized Schur-complement criterion
gives
\(\operatorname{Ran}((B^{(N)})^{\mathsf T})
 \subseteq\operatorname{Ran}C^{(N)}\). Define
\[
A_1^{(N)}
:=
B^{(N)}(C^{(N)})^\dagger(B^{(N)})^{\mathsf T}.
\]
Then
\begin{equation}
\label{eq:case-5-schur-decomposition}
M(N)=
\begin{pmatrix}
A^{(N)}-A_1^{(N)} & 0\\
0 & 0
\end{pmatrix}
+
\begin{pmatrix}
A_1^{(N)} & B^{(N)}\\
(B^{(N)})^{\mathsf T} & C^{(N)}
\end{pmatrix},
\end{equation}
and both summands are positive semidefinite. We shall show that,
after suitable truncation, the first summand corresponds to
\(\cZ_y\), while the second corresponds to \(\cZ_g\).

Set \(R:=N-\ell=n+h\). For each \(a\in\{1,\ldots,R\}\), define
\[
z_a:=X^{a-1}Y^\ell.
\]
Since \(\deg z_a=a+\ell-1\leq R+\ell-1=N-1\), the monomial \(z_a\)
belongs to \(\mathcal T_N\).

Let \(r\in\mathcal T_N\). Since \(r\) contains a positive power of
\(Y\), the quotient \(r/Y\) is a monomial, and
\[
X^ar-z_ar
=
X^{a-1}r(X-Y^\ell)
=
pX^{a-1}\frac rY.
\]
The terms on the left have degree at most \(2N\), so
\eqref{eq:case-5-saturation} gives
\(L_\beta(X^ar)=L_\beta(z_ar)\). In terms of the blocks
\(B^{(N)}\) and \(C^{(N)}\), this means
\begin{equation}
\label{eq:case-5-block-row-identity}
B^{(N)}_{X^a,\bullet}
=
C^{(N)}_{z_a,\bullet},
\qquad 1\leq a\leq R.
\end{equation}

Applying the same Schur-completion calculation as in the proof of
Theorem~\ref{theo:case-3}, and using
\(C^{(N)}(C^{(N)})^\dagger C^{(N)}=C^{(N)}\), gives
\begin{equation}
\label{eq:case-5-completed-block-entries}
(A_1^{(N)})_{X^a,X^b}
=
C^{(N)}_{z_a,z_b},
\qquad 1\leq a,b\leq R.
\end{equation}

The entries involving the constant monomial require a separate
observation. Since
\(\operatorname{Ran}((B^{(N)})^{\mathsf T})
 \subseteq\operatorname{Ran}C^{(N)}\), we have
\[
B^{(N)}_{1,\bullet}
(C^{(N)})^\dagger C^{(N)}
=
B^{(N)}_{1,\bullet}.
\]
Consequently, for \(1\leq a\leq R\),
\begin{equation}
\label{eq:case-5-constant-row-entries}
\begin{aligned}
(A_1^{(N)})_{1,X^a}
&=
B^{(N)}_{1,\bullet}(C^{(N)})^\dagger
(B^{(N)}_{X^a,\bullet})^{\mathsf T}\\
&=
B^{(N)}_{1,\bullet}(C^{(N)})^\dagger
C^{(N)}e_{z_a}\\
&=
B^{(N)}_{1,z_a}.
\end{aligned}
\end{equation}

Equations \eqref{eq:case-5-completed-block-entries} and
\eqref{eq:case-5-constant-row-entries} show that the relevant
portion of \(A_1^{(N)}\) is Hankel. Indeed,
\[
(A_1^{(N)})_{1,X^a}
=
\beta_{a-1,\ell},
\qquad
(A_1^{(N)})_{X^a,X^b}
=
\beta_{a+b-2,\,2\ell}.
\]
Moreover, \eqref{eq:case-5-saturation}, applied to
\(X^{a+b-2}Y^\ell(X-Y^\ell)\), gives
\[
\beta_{a+b-2,\,2\ell}
=
\beta_{a+b-1,\ell}.
\]
Thus \((A_1^{(N)})_{X^a,X^b}\) depends only on \(a+b\) and is
consistent with the entries
\((A_1^{(N)})_{1,X^{a+b}}\) whenever the latter are within the
available range. Hence the principal submatrix of \(A_1^{(N)}\)
indexed by \(1,X,\ldots,X^R\) is Hankel.

Let \(M^{[g]}(R)\) be the principal submatrix of the second summand in
\eqref{eq:case-5-schur-decomposition} indexed by all monomials of
degree at most \(R\). Explicitly,
\[
M^{[g]}(R)
:=
\begin{pmatrix}
A_1^{(R)} & B^{(R)}\\
(B^{(R)})^{\mathsf T} & C^{(R)}
\end{pmatrix},
\]
where the blocks denote the corresponding restrictions of those in
\eqref{eq:case-5-schur-decomposition}. As a principal submatrix of
a positive semidefinite matrix, \(M^{[g]}(R)\succeq0\). The entries of
\(B^{(R)}\) and \(C^{(R)}\) are entries of the original moment
matrix, while the preceding identities show that \(A_1^{(R)}\) is
Hankel. Consequently, \(M^{[g]}(R)\) is a well-defined bivariate moment
matrix.

We next show that its columns satisfy \(X=Y^\ell\). Since
\(z_1=Y^\ell\), \eqref{eq:case-5-block-row-identity}, with \(a=1\),
shows that the columns indexed by \(X\) and \(Y^\ell\) agree on all
rows indexed by monomials in \(\mathcal T_R\). If the row is indexed
by \(X^a\), where \(1\leq a\leq R\), then
\eqref{eq:case-5-block-row-identity} and
\eqref{eq:case-5-completed-block-entries} give
\[
(M^{[g]}(R))_{X^a,X}
=
(A_1^{(N)})_{X^a,X}
=
C^{(N)}_{z_a,z_1}
=
B^{(N)}_{X^a,z_1}
=
(M^{[g]}(R))_{X^a,Y^\ell}.
\]
For the constant row,
\eqref{eq:case-5-constant-row-entries}, with \(a=1\), gives
\[
(M^{[g]}(R))_{1,X}
=
(A_1^{(N)})_{1,X}
=
B^{(N)}_{1,z_1}
=
(M^{[g]}(R))_{1,Y^\ell}.
\]
Thus the columns indexed by \(X\) and \(Y^\ell\) coincide. Since
\(M^{[g]}(R)\succeq0\), the standard kernel property for moment matrices
shows that this relation is recursively generated.

Set \(r_0:=\max\{\ell-1,1\}\). Since \(r_0\leq h\), the restriction
of \(M^{[g]}(R)\) to order \(n+r_0\) is a positive \(r_0\)-extension of
its degree-\(2n\) truncation satisfying \(X=Y^\ell\). Equivalently,
the localizing matrices associated with \(g\) and \(-g\) vanish.
By \cite[Theorem~3.1]{z4}, this truncation admits a representing
measure \(\mu_g\) supported on
\(\cZ_g=\cZ_{x-y^\ell}\).

It remains to treat the first summand in
\eqref{eq:case-5-schur-decomposition}. Set
\[
D^{(N)}:=A^{(N)}-A_1^{(N)}.
\]
Then \(D^{(N)}\succeq0\). Moreover, \(A^{(N)}\) is Hankel, and the
preceding identities show that the principal submatrix of
\(A_1^{(N)}\) indexed by \(1,X,\ldots,X^R\) is Hankel. Therefore,
the corresponding principal submatrix of \(D^{(N)}\) is a positive
semidefinite univariate Hankel moment matrix.

Since \(h\geq2\), we have \(n+1\leq R=n+h\). Thus the principal
submatrix of \(D^{(N)}\) indexed by \(1,X,\ldots,X^{n+1}\) is a
positive semidefinite univariate Hankel moment matrix extending the
submatrix indexed by \(1,X,\ldots,X^n\). As in the axis argument in
the proof of Theorem~\ref{theo:case-3}, property \((S_{n,1})\)
gives a representing measure \(\nu\) on \(\mathbb R\) for its
degree-\(2n\) truncation. Pushing \(\nu\) forward under
\(t\mapsto(t,0)\) gives a measure \(\mu_y\) supported on \(\cZ_y\).

Restricting \eqref{eq:case-5-schur-decomposition} to moments of
degree at most \(2n\) gives
\[
\beta^{(2n)}
=
\beta_y^{(2n)}
+
\beta_g^{(2n)}.
\]
Consequently, \(\mu:=\mu_y+\mu_g\) represents \(\beta^{(2n)}\), and
\[
\operatorname{supp}\mu
\subseteq
\cZ_y\cup\cZ_g
=
\cZ_p.
\]

We have proved that the
\(\ell+\max\{\ell-1,2\}\)-extension property relative to
\(\{1,p,-p\}\) implies the existence of a
\(\cZ_p\)-representing measure. The converse follows from the
standard necessity of positive extensions and the corresponding
localizing conditions. Hence \(\{1,p,-p\}\) satisfies property
\(\left(S_{n,\ell+\max\{\ell-1,2\}}\right)\).
\end{proof}
\begin{corollary}
\label{cor-after-4-6}
Let
\[
        p(x,y)\in
        \{y(x-y^\ell),\,x(y-x^\ell)\},
        \qquad
        \ell\in\NN,
\]
and suppose that \(n\ge\ell+1\). Then every
\(f\in\mathcal P_{2n}\) that is strictly positive on \(\cZ_p\)
admits a representation
$$
        f=\sum_{j=1}^r g_j^2+ph,
$$
where \(g_1,\ldots,g_r,h\in\mathbb R[x,y]\) and
$
        g_j^2,\ ph
        \in
        \mathcal P_{
        2\left(
        n+\ell+\max\{\ell-1,2\}
        \right)}$,
        $j=1,\ldots,r.$
\end{corollary}

The proof is identical to that of
Corollary~\ref{cor-after-4-1}, using
Theorem~\ref{theo:case-5} and, for \(p=x(y-x^\ell)\), symmetry under
the interchange of \(x\) and \(y\).

\end{document}